\documentclass[12pt]{article}

\usepackage{geometry}
\usepackage{amsmath}
\numberwithin{equation}{section}
\usepackage{amssymb,amsmath}
\usepackage{cite}
\usepackage{graphicx}
\usepackage{multirow}
\usepackage{booktabs}
\usepackage{algorithm, algorithmic}
\usepackage{subfig}
\usepackage{color}
\usepackage{caption}
\usepackage{url}

\usepackage[breaklinks,colorlinks,linkcolor=black,citecolor=black,urlcolor=blue]{hyperref}

\usepackage{amssymb}

\newcommand{\R}{\mathbb{R}}

\newcommand{\stiefel}{{\cal S}_{n,p}}
\newcommand{\spp}{{\cal S}_{p,p}}

\newcommand{\symn}{\mathbb{S}^n}
\newcommand{\Rnp}{{\mathbb{R}}^{n\times p}}
\newcommand{\Rnm}{{\mathbb{R}}^{n\times m}}
\newcommand{\Rnn}{{\mathbb{R}}^{n\times n}}
\newcommand{\Rpp}{{\mathbb{R}}^{p\times p}}

\newcommand{\tr}{\mathrm{tr}}

\newcommand{\Diag}{\mathrm{Diag}}

\newcommand{\zz}{^{\top}}
\newcommand{\inv}{^{-1}}
\newcommand{\st}{\mathrm{s.\,t.}\,\,} 
\newcommand{\ff}{_{\mathrm{F}}}
\newcommand{\fs}{^2_{\mathrm{F}}}

\newcommand{\dkh}[1]{\left(#1\right)}
\newcommand{\hkh}[1]{\left\{#1\right\}}

\newcommand{\norm}[1]{\left\|#1\right\|}
\newcommand{\inner}[2]{\left\langle#1, \; #2\right\rangle}

\newcommand{\qr}[1]{\mathbf{qr}\left(#1\right)}

\newcommand{\randn}[2]{\mathbf{randn}(#1,#2)}

\newcommand{\abs}[1]{\left|#1\right|}

\usepackage{amsthm}

\newtheorem{theorem}{Theorem}[section]

\newtheorem{lemma}[theorem]{Lemma}
\newtheorem{assumption}[theorem]{Assumption}
\newtheorem{corollary}[theorem]{Corollary}
\newtheorem{remark}[theorem]{Remark}

\definecolor{Gray}{rgb}{0.5,0.5,0.5}

\title{Multipliers Correction Methods for Optimization Problems over the Stiefel Manifold}
\author{Lei Wang\thanks{State Key Laboratory of Scientific and Engineering 
		Computing, Academy of Mathematics and Systems Science, Chinese  Academy of 
		Sciences, and University of Chinese Academy of Sciences, China 
		(wlkings@lsec.cc.ac.cn). Research is supported by the National Natural Science 
		Foundation of China (No. 11971466).}
	\and 
	Bin Gao\thanks{Institute of Information and Communication Technologies, Electronics and Applied Mathematics,
		Universit\'{e} catholique de Louvain, Belgium
		ICTEAM institute at UCLouvain
		(bin.gao@uclouvain.be). Research is supported in part by the Fonds de la Recherche Scientifique -- FNRS 
		and the Fonds Wetenschappelijk Onderzoek -- Vlaanderen under EOS Project (No. 30468160).}
	\and Xin Liu\thanks{State Key Laboratory of Scientific and Engineering 
		Computing, Academy of Mathematics and Systems Science, Chinese Academy of 
		Sciences, 
		and University of Chinese Academy of Sciences, China (liuxin@lsec.cc.ac.cn). 
		Research is supported in part by the National Natural Science Foundation of 
		China (No. 11991021, 11991020 and 11971466), Key Research Program of Frontier 
		Sciences, Chinese Academy of Sciences (No. ZDBS-LY-7022), the National Center 
		for Mathematics and Interdisciplinary Sciences, Chinese Academy
		of Sciences and the Youth Innovation Promotion Association, Chinese Academy
		of Sciences.}
}

\date{}

\begin{document}
\maketitle

\begin{abstract}
	We propose a class of multipliers correction methods to minimize 
	a differentiable function over the Stiefel manifold. 
	The proposed methods combine a function value reduction step with a proximal correction step. 
	The former one searches along an arbitrary descent direction in the Euclidean space 
	instead of a vector in the tangent space of the Stiefel manifold. 
	Meanwhile, the latter one minimizes a first-order proximal approximation of the objective 
	function in the range space of the current iterate
	to make Lagrangian multipliers associated with orthogonality constraints
	symmetric at any accumulation point.
	The global convergence has been established for the proposed methods.
	Preliminary numerical experiments demonstrate that the
	new methods significantly outperform other 
	state-of-the-art first-order approaches in solving various kinds of testing problems.
\end{abstract}



\section{Introduction}

\label{sec:introduction}
	
	We focus on the matrix-variable optimization problems with orthogonality constraints:
	\begin{equation}
		\begin{aligned}
			\min\limits_{X \in \Rnp} \hspace{2mm} & f (X) \\
			\st  \hspace{2.5mm} & X\zz X = I_p,
		\end{aligned}
		\label{eq:ORTH}
	\end{equation}
	where $p \leq n$, $I_p$ is the $p \times p$ identity matrix, 
	and $f: \Rnp \longrightarrow \R$ is a continuously differentiable function. 
	The feasible region, 
	denoted by $\stiefel := \hkh{X \in \Rnp \mid X\zz X = I_p }$,
	is called the Stiefel manifold.
	
	Optimization problems over the Stiefel manifold
	have wide applications in
	scientific computing and data science.
	For example, in linear eigenvalue problems \cite{Caboussat2009,Liu2013,Liu2015b}, 
	energy minimization in electronic structure calculations \cite{Yang2007,Liu2014,Liu2015a},
	matrix completion \cite{Boumal2015},
	independent component analysis \cite{Sato2017},
	Bose--Einstein condensates \cite{Wu2017},
	 discriminant analysis \cite{Li2017},
	dictionary learning \cite{Hu2020b},
	and nearest low-rank correlation matrix problems \cite{Grubisic2007}.
	Beyond that, one can find other applications in \cite{Edelman1998,Absil2008} and the references therein.

\subsection{Existing works}

	Optimization problems over the Stiefel manifold 
	have been adequately studied in recent decades. 
	There emerge quite a few algorithms and solvers, such as, 
	geodesic--based approaches \cite{Edelman1998,Manton2002,Nishimori2005},
	retraction--based approaches \cite{Absil2006,Yang2007,Abrudan2008,Abrudan2009,
		Savas2010,Wen2013,Absil2012,Jiang2013,Huang2015},
	and splitting and alternating approaches \cite{Lai2014,Rosman2014}.
	We refer the interested readers to the monograph \cite{Absil2008} and survey \cite{Hu2020} on these methods.
	Recently, the authors in \cite{Gao2019} developed two orthonormalization-free approaches, 
	called PLAM and PCAL, 
	which are based on the augmented Lagrangian penalty function \cite{Noceda2006}
	but adopt an explicit expression to update Lagrangian multipliers instead of the dual ascent step.
	Such approaches are particularly suitable for parallel computing due to their high scalability.
	PCAL was further applied  to solve the energy minimization problem in electronic structure calculations \cite{Gao2020}. 
	More recently, an exact penalty model, which shares the same global minimizers as the original problem~\eqref{eq:ORTH}, 
	was proposed in \cite{Xiao2020}. 
	In order to solve this model, they also proposed first-order and second-order approaches which subsume PCAL as a specific implementation.
	

	In \cite{Gao2018}, the authors proposed a new algorithmic framework which consists of two steps:
	the function value reduction step, which preserves the feasibility, is conducted in the Euclidean space;
	the correction step is nothing but a rotation on the previously obtained step.
	As the Lagrangian multipliers associated with orthogonality constraints are symmetric 
	and enjoy an explicit expression $X\zz \nabla f(X)$ 
	at any first-order stationary point of \eqref{eq:ORTH} (see \cite[(2.2)]{Gao2019}),
	the purpose of this correction step is to guarantee the symmetry of $X\zz \nabla f(X)$ at each iteration.
	In summary, three algorithms were introduced in \cite{Gao2018} to fulfill the framework; 
	extensive numerical results illustrated their great potential. 
	However, this framework strictly depends on the following assumption.
	
	\begin{assumption}
		\label{asp:Gao-2}
		$f (X) = h (X) + \tr\dkh{G\zz X}$,
		where $G \in \Rnp$ is a constant matrix and $h (X)$ is orthogonal invariant, 
		i.e., $h(XQ) = h(X)$ holds for any $Q \in \spp$. 
		Moreover, $\nabla h(X) = H(X) X$, where $H:\Rnp \longrightarrow \symn$ 
		and $\symn$ refers to the set of $n \times n$ symmetric matrices. 
	\end{assumption}
	
	Assumption \ref{asp:Gao-2} restricts the objective to a class of composite functions.
	In this case, the explicit expression $X\zz \nabla f (X)$ can be divided into two parts, 
	including a symmetric term $X\zz H(X) X$ and a linear term $X\zz G$. 
	Hence, it is sufficient to guarantee the symmetry of $X\zz \nabla f(X)$ in the correction step 
	by making $X\zz G$ symmetric.
	To this end, one can minimize $\tr \dkh{G\zz X}$ in the range space of $X$ 
	whose finding its global minimizer is equivalent to computing a singular value decomposition.
	
	Although quite a few practical problems---such as linear eigenvalue problem and
	energy minimization in electronic structure calculations---satisfy this assumption,
	there exist important scenarios in which Assumption~\ref{asp:Gao-2} does not hold; 
	e.g., minimizing the Brockett function (weighted sum of eigenvalues) \cite{Anstreicher2000,Absil2008},
	 joint diagonalization problems \cite{Sato2017},
	 and dictionary learning \cite{Hu2020b} over the Stiefel manifold.

\subsection{Motivation and contribution}

	In this paper, we intend to address the restriction of Assumption~\ref{asp:Gao-2}. 
	Specifically, we solve optimization problems over the Stiefel manifold with a general objective function. 
	To this end, we propose multipliers correction algorithmic framework,
	and it contains two steps. 
	The first step is to minimize the objective function in the Euclidean space. 
	Gradient reflection, gradient projection and column-wise block coordinate descent algorithms proposed in \cite{Gao2018} 
	are similarly introduced in this step.
	Then we propose a novel multipliers correction step
	whose essential idea is to minimize 
	a first-order proximal approximation of the objective function in the range space of the current iterate.
	The main computational cost of such correction step is 
	calculating the singular value decomposition of a $p\times p$ matrix, 
	which shares the same cost with the correction step introduced in \cite{Gao2018}.
	This correction step can further reduce the function value
	and guarantee the symmetry of Lagrangian multipliers at any accumulation point. 
	Remarkably, the new methods work for a much wider range of problems than those proposed in \cite{Gao2018}.
	
	In addition, we prove the global convergence and worst case complexity of the proposed methods.
	Numerical experiments illustrate their effectiveness. 
	Note that the new methods outperform some state-of-the-art first-order algorithms for optimization over the Stiefel manifold, 
	and also work well in those instances which are out of the scope of the algorithms proposed in \cite{Gao2018}.

\subsection{Notation}

	The Euclidean inner product of two matrices $Y_1\in\Rnm$ and $Y_2\in\Rnm$ 
	is defined as $\inner{Y_1}{Y_2}=\tr\dkh{ Y_1\zz Y_2 }$, 
	where $\tr(B)$ is the trace of a square matrix $B \in \R^{m \times m}$.
	The Frobenius norm and 2-norm of a matrix $C \in \Rnm$ 
	are denoted by $\norm{C}\ff$ and $\norm{C}_2$, respectively. 
	We use $C^{\dagger}$ to represent the pseudo-inverse of $C$. 
	$C_i$ and $C_{ij}$ stand for the $i$-th column and $(i,j)$-th element of $C$, respectively. 
	$C_{\bar{i}} \in \R^{n \times (m - 1)}$ refers to the matrix $C$ removing its $i$-th column, 
	namely, $C_{\bar{i}} = [C_1, \; \dotsc, \; C_{i-1}, \; C_{i+1}, \; \dotsc, \; C_{m}]$. 
	$C_{i,v} \in \Rnm$ stands for the matrix whose $i$-th column of $C$ is replaced with a vector $v \in \R^n$, 
	i.e., $C_{i,v}=[C_1, \; \dotsc, \; C_{i-1}, \; v, \; C_{i+1}, \; \dotsc, \; C_{m}]$.
	The ball centered at $C\in \Rnm$ with radius $r > 0$ 
	is denoted by ${\cal B}(C, r) = \hkh{P \in \Rnm \mid \norm{P - C}\ff \leq r}$. 
	$\qr{C}$ refers to the Q-matrix of reduced QR decomposition of $C$. 
	The projection of a matrix $W \in \Rnp$ to the Stiefel manifold $\stiefel$ 
	is denoted by ${\cal P}_{\stiefel} (W)$. 
	$\Diag(\xi) \in \Rnn$ denotes the diagonal matrix with entries of $\xi \in \R^n$ in its diagonal. 

\subsection{Organization}

	The rest of this paper is organized as follows. 
	In Section \ref{sec:pm}, we introduce our multipliers correction methods. 
	Then we establish the theoretical analysis in Section \ref{sec:ca}. 
	Furthermore, numerical experiments are presented in Section \ref{sec:ne}. 
	In the end, we summarize this paper in Section \ref{sec:c}.

\section{Multipliers correction method}

\label{sec:pm}
	
	In this section, we present the framework of our new approaches. 
	We start with the first-order optimality condition of the optimization problem over the Stiefel manifold \eqref{eq:ORTH}.
	According to \cite[Lemma 2.2]{Gao2018}, 
	a point $X \in \Rnp$ is a first-order stationary point of \eqref{eq:ORTH},
	if and only if it satisfies the following equalities:
	\begin{equation}
		\label{eq:OC}
		\left\{
		\begin{aligned}
			& (I_n - XX\zz) \nabla f (X) = 0, \\
			& X\zz\nabla f(X) = \nabla f (X)\zz X, \\
			& X\zz X = I_p.                       
		\end{aligned}
		\right.
	\end{equation}
	The first equality in \eqref{eq:OC} stands for the stationarity 
	of the gradient in the null space of $X\zz$.
	The second equality determines the symmetry of
	Lagrangian multipliers associated with orthogonality constraints.
	For convenience, we call these three equalities ``sub-stationarity", ``symmetry" and ``feasibility", respectively.
	
	In order to solve the problem~\eqref{eq:ORTH}, 
	we adopt the similar algorithmic framework proposed in \cite{Gao2018}, 
	which consists of two steps:
	reduce the function value in proportion to the ``sub-stationarity" violation 
	and preserve the ``symmetry". 
	During the calculations of these two steps, we maintain the ``feasibility" all the time.
	
	In Subsection \ref{subsec:first}, we first review the function value reduction step in \cite{Gao2018} 
	based on Assumption~\ref{asp:Gao-1} on the differentiability of the objective function.
	Then, in Subsection \ref{subsec:second},
	we introduce a new proximal correction strategy, 
	which can further reduce the function value in proportion to the ``symmetry'' violation. 
	In the end, we present the complete algorithmic framework in Subsection~\ref{subsec:complete}.
	
	\begin{assumption}
		\label{asp:Gao-1}
		$f (X)$ is twice differentiable. 
		Then we can define $\rho \geq 0$ as 
		\begin{equation*}
		\label{eq:rho}
			\rho := \sup\limits_{X \in \tilde{\cal S}} \norm{ \nabla^2 f (X) }_2,
		\end{equation*}
		where $\tilde{\cal S} = \{ Y \in  \Rnp \mid \norm{Y}\fs < p + 1 \}$. 
		In fact, $\tilde{\cal S}$ can be replaced by any given bounded open set which contains $\stiefel$.
	\end{assumption}

\subsection{Function value reduction step}

\label{subsec:first}

	Let $X^{(k)} \in \stiefel$ be the current iterate.
	The function value reduction step is trying to 
	find a feasible intermediate point $\bar{X}^{(k)} \in \stiefel$ 
	satisfying the following sufficient function value reduction condition:
	\begin{equation}
		\label{eq:prox-first}
		f ( X^{(k)} ) - f ( \bar{X}^{(k)} ) \geq c_1 \norm{ \dkh{ I_n - X^{(k)}(X^{(k)})\zz } \nabla f ( X^{(k)}) }\fs,
	\end{equation}
	where $c_1>0$ is a constant. 
	The right hand side of \eqref{eq:prox-first}
	is in proportion to 
	the squared Frobenius norm of ``sub-stationary" violation at $X^{(k)}$. 
	Note that it can also be viewed as the projected gradient at $X^{(k)}$ in the Euclidean space. 
	In \cite{Gao2018}, the authors introduce three algorithms to 
	achieve the sufficient function value reduction \eqref{eq:prox-first}. 
	We list them below.
	
	{\bf Gradient reflection (GR) method.}
	It takes the reflection point of the current iterate $X^{(k)}$
	on the null space of $X^{(k)} - \tau\nabla f(X^{(k)})$, 
	which can be calculated by the Householder transformation.
	\begin{equation*}
		\left\{
		\begin{aligned}
			&V = X^{(k)} - \tau \nabla f (X^{(k)}) \mbox{~~for a fixed chosen~~} \tau \in (0, \rho^{-1}), \\
			&\bar{X}_{\text{GR}}^{(k)} = \dkh{- I_n + 2V(V\zz V)^{\dagger} V\zz} X^{(k)}.
		\end{aligned}
		\right.
	\end{equation*}
	
	{\bf Gradient projection (GP) method.}
	It directly projects $X^{(k)} - \tau \nabla f (X^{(k)})$ onto the Stiefel manifold, 
	which can be calculated by the following projection.
	\begin{equation*}
		\left\{
		\begin{aligned}
			&V = X^{(k)} - \tau \nabla f (X^{(k)}) \mbox{~~for a fixed chosen~~} \tau \in (0, \rho^{-1}), \\
			&\bar{X}_{\text{GP}}^{(k)} = {\cal P}_{\stiefel}(V).
		\end{aligned}
		\right.
	\end{equation*}
	Indeed, the projection ${\cal P}_{\stiefel}$ is equivalent to the singular value decomposition, 
	namely, ${\cal P}_{\stiefel}(W) = RT\zz$, where $W = RST\zz$ is the reduced singular value decomposition of $W$.
	
	{\bf Column-wise block coordinate descent (CBCD) method.}
	We minimize the objective function with respect to the $i$-th column of the variable $X$,
	and keep the remaining $p-1$ columns fixed as those $X$.
	Specifically, we sequentially solve the following subproblem:
	\begin{equation}
		\label{eq:CBCD}
		\begin{aligned}
			\min\limits_{x\in \R^{n}} \hspace{2mm} & f_{i,X}(x) := f(X_{i,x}) \\
			\st  \hspace{0.5mm} & \norm{ x }_2 = 1, \\
			& X_{\bar{i}}\zz x = 0.
		\end{aligned}
	\end{equation}
	The detailed procedure is described in Algorithm~\ref{alg:CBCD}.
	
	\begin{algorithm}
		\caption{Column--wise block coordinate descent method.}
		\begin{algorithmic}[1]
			\STATE Set $W^{(0)} = X^{(k)}$ and $i=1$.
			\WHILE{$i \leq p$}
			\STATE Solve the subproblem \eqref{eq:CBCD} with $X$ replaced by $W^{(i - 1)}$, and obtain a feasible point $x^+$ satisfying the following sufficient function value descent and asymptotic small step size safeguard:
			\begin{equation*}
				f_{i,W^{(i - 1)}} (X_i^{(k)}) - f_{i,W^{(i - 1)}}(x^+) \geq k_1 \left\| X_i^{(k)} - x^+ \right\|_2^2,
			\end{equation*}
			\begin{equation*}
				\left\| X_i^{(k)}-x^+ \right\|_2 \geq k_2 \left\| \left(I_n-W^{(i - 1)}(W^{(i - 1)})\zz\right) \nabla f_{i,W^{(i - 1)}} (X_i^{(k)}) \right\|_2,
			\end{equation*}
			where $k_1>0$ and $k_2>0$ are constants.
			\STATE Set $W^{(i)} = W^{(i - 1)}_{i,x^+}$ and $i \leftarrow i+1$.
			\ENDWHILE
			\STATE Return $\bar{X}^{(k)}_{\text{CBCD}}=W^{(p)}$.
		\end{algorithmic}
		\label{alg:CBCD}
	\end{algorithm}
	
	According to \cite[Lemmas 3.2, 3.3 and 3.8]{Gao2018}, 
	these three methods---GR, GP and CBCD---provide an intermediate point 
	satisfying the sufficient function value reduction condition~\eqref{eq:prox-first}.

\subsection{Proximal correction step}

\label{subsec:second}

	The intermediate point $\bar{X}^{(k)} \in \stiefel$ 
	obtained in the previous subsection
	does not necessarily satisfy the ``symmetry" equality
	$(\bar{X}^{(k)})\zz \nabla f(\bar{X}^{(k)}) = \nabla f(\bar{X}^{(k)})\zz \bar{X}^{(k)}$ 
	in \eqref{eq:OC}. 
	In \cite{Gao2018}, the authors 
	introduce a correction step to obtain $X^{(k+1)}$
	through a rotation on $\bar{X}^{(k)}$.
	The validity of this correction step highly depends on 
	Assumption \ref{asp:Gao-2},
	and can not be extended to the general case.
	
	In order to address this issue, we introduce a new proximal strategy. 
	We still 
	calculate the next iterate by a rotation 
	$X^{(k+1)} = \bar{X}^{(k)}Q$ with $Q \in \spp$.
	Ideally, we expect a minimization on $f(\bar{X}^{(k)}Q)$
	and desire to satisfy the ``symmetry" equality for $X^{(k+1)}$.
	However, it is intractable to ``cheaply'' minimize a general objective function 
	in the range space of $\bar{X}^{(k)}$:
	\begin{equation}
		\label{prob:rotation}
		\min_{Q \in \spp} \hspace{2mm} f (\bar{X}^{(k)}Q).
	\end{equation}
	On the other side, even if a global solution $Q^\ast$ of \eqref{prob:rotation} is obtained,  
	the corresponding $X^{(k+1)} = \bar{X}^{(k)}Q^\ast$ 
	does not necessarily satisfy the ``symmetry" equality in general.
	
	To this end, we replace the objective function $f (X)$ with its proximal linear approximation
	$\tilde{f}(X)$ at $\bar{X}^{(k)}$ in the problem~\eqref{prob:rotation}, where
	\begin{equation*}
			\tilde{f}(X) := f(\bar{X}^{(k)}) + \inner{ \nabla f(\bar{X}^{(k)}) }{ X - \bar{X}^{(k)} }
			+\dfrac{\gamma}{2} \norm{ X - \bar{X}^{(k)} }\fs,
	\end{equation*}
	and $\gamma > 0$ is a proximal parameter. 
	Accordingly, we can construct the approximation problem:
	\begin{equation}
		\label{prob:rotation-approx}
		\min_{Q \in \spp} \hspace{2mm} \tilde{f}(\bar{X}^{(k)}Q).
	\end{equation}
	In view of the orthogonality of $\bar{X}^{(k)}$ and $Q$, 
	it is straightforward to obtain the following equivalent problem for \eqref{prob:rotation-approx}:
	\begin{equation}
		\label{eq:epfindQ}
		\min_{Q \in \spp} \hspace{2mm} g(Q) := \tr \dkh{ Q\zz Z^{(k)} },
	\end{equation}
	where $Z^{(k)} := (\bar{X}^{(k)})\zz \nabla f(\bar{X}^{(k)}) - \gamma I_p$.  
	If $Z^{(k)} = 0$, the problem~\eqref{eq:epfindQ} is trivial 
	and we choose $X^{(k+1)} =\bar{X}^{(k)}$.
	Otherwise, it is known that the global solution of \eqref{eq:epfindQ} is 
	\begin{equation*}
		Q^{(k)} := -U V\zz,
	\end{equation*}
	where $U \in \Rpp$ and $V \in \Rpp$ come from the singular value decomposition
	$Z^{(k)}=U \Sigma V\zz$. In summary, we can construct a new iterate as follows,
	\begin{equation}
		\label{eq:prox-correct}
		X^{(k+1)}=
		\left\{
		\begin{array}{ll}
			\bar{X}^{(k)}, & \mbox{if~} (\bar{X}^{(k)})\zz \nabla f(\bar{X}^{(k)}) = \gamma I_p; \\
			\bar{X}^{(k)}Q^{(k)},
			& \mbox{otherwise.}
		\end{array}
		\right.
	\end{equation}

	We call \eqref{eq:prox-correct} the proximal correction step. 
	This step can further reduce the objective function value in proportion to 
	$\norm{ (\bar{X}^{(k)})\zz \nabla f(\bar{X}^{(k)}) - \nabla f(\bar{X}^{(k)})\zz \bar{X}^{(k)} }\fs$, 
	which will be proved in Section \ref{sec:ca}.

\subsection{Complete algorithmic framework}

\label{subsec:complete}

	We denote
	\begin{equation*}
		c(X) := \nabla f(X) - X \nabla f(X)\zz X.
	\end{equation*}
	Note that it measures the stationarity violation of \eqref{eq:OC} which represents the combination
	of ``sub-stationarity" violation and ``symmetry"  violation since 
	\begin{equation}
		\norm{ c(X) }\fs = \norm{ \dkh{I_n - XX\zz } \nabla f(X) }\fs + \norm{ X\zz\nabla f(X) - \nabla f(X)\zz X }\fs
		\label{eq:ssc}
	\end{equation}
	holds for any $X\in\stiefel$. 
	The complete algorithmic framework is described in Algorithm~\ref{alg:MCM}. 
	
	\begin{algorithm}
		\caption{Multipliers correction methods.}
		\begin{algorithmic}[1]
			
			\STATE Set tolerance $\epsilon>0$, proximal parameter $\gamma > 0$,
			and initial point $X^{(0)} \in\stiefel$; Set $k\leftarrow0$.
			
			\WHILE{$\|c(X^{(k)})\|\ff>\epsilon$}
			
			\STATE Based on $X^{(k)}$, find a feasible point $\bar{X}^{(k)}$ satisfying~\eqref{eq:prox-first};
			
			\STATE Based on $\bar{X}^{(k)}$, compute $X^{(k+1)}$ by \eqref{eq:prox-correct};
			
			
			\STATE Set $k \leftarrow k+1$;
			
			\ENDWHILE
			
			\STATE Return $X^{(k)}$.
			
		\end{algorithmic}
		\label{alg:MCM}
	\end{algorithm}

	As $X\zz \nabla f(X)$ is nothing but the explicit expression 
	of Lagrangian multipliers associated with orthogonality constraints 
	at any first-order stationary point of \eqref{eq:ORTH}, 
	we call our framework applying the proximal correction step as the multipliers correction methods (MCM).
	For the algorithms taking GR, GP and CBCD in the Step 3 of Algorithm~\ref{alg:MCM}, 
	we call them GRP, GPP and CBCDP, respectively.

\section{Convergence analysis}

\label{sec:ca}

	In this section, we establish the global convergence and worst case complexity of Algorithm~\ref{alg:MCM}.
	First of all, using the compactness of $\stiefel$, we can define the following two constants.
	\begin{equation*}
		\underline{f} := \min\limits_{X\in\stiefel} f(X), \qquad M := \max\limits_{X \in \stiefel} \norm{ \nabla f(X) }_2.
	\end{equation*}
	
	Now we evaluate the sufficient function value reduction in the multipliers correction step.
	
	\begin{lemma}
		\label{le:fvdtilde}
		Suppose Assumption \ref{asp:Gao-1} holds and $\gamma > \rho$.
		Let $\bar{X}^{(k)}\in\stiefel$ and $X^{(k+1)}$ be computed by \eqref{eq:prox-correct}.
		Then we have $X^{(k+1)}\in\stiefel$. 
		In addition, it holds that 
		\begin{equation}
			\label{eq:fvdtildex}
			f(\bar{X}^{(k)}) - f(X^{(k+1)})
			\geq \dfrac{1}{8 c_{\gamma}}\norm{ (\bar{X}^{(k)})\zz \nabla f(\bar{X}^{(k)})-\nabla f(\bar{X}^{(k)})\zz \bar{X}^{(k)} }\fs,
		\end{equation}
		where $c_{\gamma} = M + \gamma > 0$ is a constant.
	\end{lemma}
	
	\begin{proof}
		The feasibility $X^{(k+1)} \in \stiefel$ is obvious. 
		Next, we only focus on the inequality~\eqref{eq:fvdtildex}. 
		If $Z^{(k)} = 0$, we have $X^{(k+1)} =\bar{X}^{(k)}$ 
		and $(\bar{X}^{(k)})\zz \nabla f(\bar{X}^{(k)}) = \gamma I_p$ is symmetric, 
		which implies \eqref{eq:fvdtildex} immediately. 
		Otherwise, since $\gamma > \rho$, we can use Taylor's Theorem and obtain
		\begin{equation*}
			f(X^{(k+1)}) 
			\leq 
			f(\bar{X}^{(k)}) + \inner{ \nabla f (\bar{X}^{(k)})}{ X^{(k+1)} - \bar{X}^{(k)} } + \dfrac{\gamma}{2} \norm{ X^{(k+1)}-\bar{X}^{(k)} }\fs.
		\end{equation*}
		Due to the updating rule \eqref{eq:prox-correct} and decomposition 
		$Z^{(k)} = (\bar{X}^{(k)})\zz \nabla f (\bar{X}^{(k)}) - \gamma I_p = U \Sigma V\zz$, 
		we have
		\begin{equation}
			\label{eq:funs}
			\begin{aligned}
				f (\bar{X}^{(k)}) - f (X^{(k+1)}) 
				\geq {} & - \inner{ (\bar{X}^{(k)})\zz \nabla f(\bar{X}^{(k)}) } { Q^{(k)} - I_p } 
				- \dfrac{\gamma}{2} \norm{ Q^{(k)} - I_p }\fs \\
				= {} & \tr \dkh{ \Sigma } - \gamma \tr \dkh{ Q^{(k)} } 
				+ \tr \dkh{ (\bar{X}^{(k)})\zz \nabla  f(\bar{X}^{(k)}) }
				- \gamma p + \gamma\tr \dkh{ Q^{(k)} } \\
				= {} & \tr \dkh{ \Sigma + U \Sigma V\zz }.
			\end{aligned}
		\end{equation}
		Let $\hat{\Sigma} = \Sigma V\zz U$ 
		and $\Gamma = ( \hat{\Sigma} + \hat{\Sigma}\zz ) / 2$. 
		It is easy to show that $\tr \dkh{ U \Sigma V\zz } = \tr \dkh{ \Gamma }$. 
		On the other side, after simple calculations, we can obtain that
		\begin{equation}
			\label{eq:rhs}
			\begin{aligned}
				\norm{ (\bar{X}^{(k)})\zz \nabla f (\bar{X}^{(k)}) - \nabla f (\bar{X}^{(k)})\zz\bar{X}^{(k)} }\fs
				= {} & \norm{ U \Sigma V\zz - V \Sigma U\zz }\fs \\
				= {} & 2 \tr \dkh{\Sigma^2} - 2 \tr \dkh{ \Sigma V\zz U \Sigma V\zz U } \\
				= {} & 2 \tr \dkh{ \Sigma^2 } - 2\tr \dkh{ \hat{\Sigma}^2 }. 
			\end{aligned}
		\end{equation}
		It follows from the equality $\Gamma = \dkh{ \hat{\Sigma} + \hat{\Sigma}\zz } / 2$
		that 
		$2 \tr \dkh{ \Gamma^2 } = \tr \dkh{ \Sigma^2 } + \tr \dkh{ \hat{\Sigma}^2 }$.
		Together with \eqref{eq:rhs}, we arrive at
		\begin{equation}
			\label{eq:sym-barx}
			\norm{ (\bar{X}^{(k)})\zz \nabla f (\bar{X}^{(k)}) - \nabla f(\bar{X}^{(k)})\zz\bar{X}^{(k)} }\fs 
			= 4 \tr \dkh{ \Sigma^2 } - 4 \tr \dkh{ \Gamma^2 }.
		\end{equation}
		Moreover, we have
		$\tr \dkh{ \Gamma^2 } = \tr \dkh{ \Gamma\zz \Gamma } = \sum\limits_{i = 1}^{p} \Gamma_i\zz \Gamma_i 
		\geq \sum\limits_{i = 1 }^{p} \Gamma_{ii}^2$.
		Hence, it holds that
		\begin{equation*}
			\tr\dkh{ \Sigma^2 } - \tr \dkh{ \Gamma^2 } 
			\leq \sum\limits_{i = 1}^{p}\dkh{ \Sigma_{ii}^2 - \Gamma_{ii}^2 }
			=\sum\limits_{i = 1}^{p} \dkh{ \Sigma_{ii} - \Gamma_{ii} } \dkh{\Sigma_{ii} + \Gamma_{ii} }.
		\end{equation*}
		According to the definition of $\Gamma$, we can obtain
		$\abs{ \Gamma_{ii} } = \Sigma_{ii} \dkh{ V_i\zz U_i } \leq \Sigma_{ii} \norm{ V_i }_2 \norm{ U_i }_2 = \Sigma_{ii}$,
		which implies
		\begin{equation}
			\label{eq:grad}
			\tr \dkh{ \Sigma^2 } - \tr \dkh{ \Gamma^2 } 
			\leq \sum\limits_{i = 1}^{p} 2 \Sigma_{ii} \dkh{ \Sigma_{ii} + \Gamma_{ii} } 
			\leq 2 \norm{ \Sigma }_2 \tr \dkh{ \Sigma + \Gamma }
			\leq 2 c_{\gamma} \tr \dkh{ \Sigma + \Gamma },
		\end{equation}
		where the last inequality follows from 
		$\norm{ \Sigma }_2 = \norm{ Z^{(k)} }_2 \leq \norm{ (\bar{X}^{(k)})\zz \nabla f(\bar{X}^{(k)})}_2 + \gamma 
		\leq M+\gamma = c_\gamma$.
		Combing \eqref{eq:sym-barx} and \eqref{eq:grad}, we can deduce that
		\begin{equation}
			\label{eq:des-barx}
			8 c_{\gamma} \tr \dkh{ \Sigma + U \Sigma V\zz } 
			= 8 c_{\gamma} \tr \dkh{ \Sigma + \Gamma } 
			\geq \norm{ (\bar{X}^{(k)})\zz \nabla f (\bar{X}^{(k)}) - \nabla f(\bar{X}^{(k)})\zz\bar{X}^{(k)} }\fs ,
		\end{equation}
		which together with \eqref{eq:funs} infers that
		\begin{equation*}
			8 c_{\gamma} \dkh{ f (\bar{X}^{(k)}) - f (X^{(k+1)}) }
			\geq \norm{ (\bar{X}^{(k)})\zz \nabla f (\bar{X}^{(k)}) - \nabla f (\bar{X}^{(k)})\zz \bar{X}^{(k)} }\fs.
		\end{equation*}
		This completes the proof.
	\end{proof}

	The convergence of the function value can be a direct corollary.
	
	\begin{corollary}
		\label{cor:fvc}
		Suppose Assumption \ref{asp:Gao-1} holds, $\gamma > \rho$, and $\{ X^{(k)} \}$ is the iterate sequence generated 
		by Algorithm~\ref{alg:MCM}. 
		Then $\{ f(X^{(k)}) \}$ is convergent. 
	\end{corollary}

	\begin{proof}
		According to Lemma \ref{le:fvdtilde}, we have
		\begin{equation}
			\label{eq:fvdk+1}
			\begin{aligned}
				& f (X^{(k)}) - f(X^{(k+1)}) 
				= f (X^{(k)}) - f (\bar{X}^{(k)}) + f (\bar{X}^{(k)}) - f (X^{(k+1)}) \\
				\geq {} & c_1 \norm{ \dkh{ I_n - X^{(k)}(X^{(k)})\zz } \nabla f (X^{(k)}) }\fs 
				+ \dfrac{1}{8c_{\gamma}} 
				\norm{ (\bar{X}^{(k)})\zz \nabla f (\bar{X}^{(k)}) - \nabla f (\bar{X}^{(k)})\zz \bar{X}^{(k)} }\fs \\
				\geq {} & c_1 \norm{ \dkh{I_n - X^{(k)}(X^{(k)})\zz } \nabla f (X^{(k)}) }\fs
				\geq 0.
			\end{aligned}
		\end{equation}
		Consequently, $\{ f(X^{(k)}) \}$ is a monotonically non-increasing sequence. 
		On the other hand, it follows from the  compactness of the Stiefel manifold $\stiefel$ 
		that $\{ f (X^{(k)}) \}$ has a lower bound $\underline{f}$.
		Therefore, we conclude that $\{ f(X^{(k)}) \}$ is convergent, 
		which completes the proof.
	\end{proof}
	
	Then we show that the ``symmetry" violation can be controlled by the distance between $X^{(k+1)}$ and $\bar{X}^{(k)}$.
	
	\begin{lemma}
		\label{le:symk+1}
		Suppose Assumption \ref{asp:Gao-1} holds
		and $\{ X^{(k)} \}$ is the iterate sequence generated by Algorithm~\ref{alg:MCM}.
		Then it can be verified that
		\begin{equation}
			\label{eq:symk+1}
			\norm{ (X^{(k+1)})\zz \nabla f(X^{(k+1)}) - \nabla f(X^{(k+1)})\zz X^{(k+1)} }\ff 
			\leq 2 \dkh{ \rho + \gamma } \norm{ X^{(k+1)} - \bar{X}^{(k)} }\ff.
		\end{equation}
	\end{lemma}
	
	\begin{proof}
		If $Z^{(k)}= 0$, we have $X^{(k+1)} = \bar{X}^{(k)}$.
		Hence, the matrix 
		$(X^{(k+1)})\zz \nabla f(X^{(k+1)}) = (\bar{X}^{(k)})\zz \nabla f(\bar{X}^{(k)}) = \gamma I_p$ 
		is symmetric,
		which infers \eqref{eq:symk+1} immediately.
		Next, we investigate the case that $Z^{(k)} \neq 0$.
		It follows from the definition of $Q^{(k)} = - U V\zz$ and decomposition $Z^{(k)} = U \Sigma V\zz$ 
		that $(Q^{(k)})\zz Z^{(k)} = (Z^{(k)})\zz Q^{(k)}$.
		In view of $Z^{(k)} = (\bar{X}^{(k)})\zz \nabla f (\bar{X}^{(k)}) - \gamma I_p$ 
		and $X^{(k+1)}=\bar{X}^{(k)} Q^{(k)}$,
		it further holds that 
		\begin{equation*}
			(X^{(k+1)})\zz \nabla f (\bar{X}^{(k)}) - \nabla f (\bar{X}^{(k)})\zz X^{(k+1)}
			= \gamma (X^{(k+1)})\zz \bar{X}^{(k)} - \gamma (\bar{X}^{(k)})\zz X^{(k+1)}.
		\end{equation*}
		According to the triangular inequality, we have
		\begin{equation*}
			\begin{aligned}
			& \norm{ (X^{(k+1)})\zz \bar{X}^{(k)} - (\bar{X}^{(k)})\zz X^{(k+1)} }\ff \\
			\leq {} & \norm{ (X^{(k+1)})\zz \bar{X}^{(k)} - (\bar{X}^{(k)})\zz \bar{X}^{(k)} }\ff 
			+ \norm{ (\bar{X}^{(k)})\zz \bar{X}^{(k)} - (\bar{X}^{k})\zz X^{(k+1)} }\ff  \\
			\leq {} & \norm{ X^{(k+1)} - \bar{X}^{(k)} }\ff \norm{\bar{X}^{(k)} }_2
			+ \norm{\bar{X}^{(k)} }_2 \norm{ X^{(k+1)} - \bar{X}^{(k)}}\ff 
			= 2 \norm{ X^{(k+1)} - \bar{X}^{(k)} }\ff,
			\end{aligned}
		\end{equation*}
		which immediately implies that
		\begin{equation}
			\label{eq:temp-1}
				\norm{ (X^{(k+1)})\zz \nabla f(\bar{X}^{(k)})-\nabla f(\bar{X}^{(k)})\zz X^{(k+1)} }\ff
				\leq 2 \gamma \norm{ X^{(k+1)} - \bar{X}^{(k)} }\ff. 
		\end{equation}
		
		On the other hand, according to Assumption \ref{asp:Gao-1}, it follows that
		\begin{equation*}
			\norm{ \nabla f (Y_1) - \nabla f (Y_2) }\ff \leq \rho \norm{ Y_1 - Y_2 }\ff, \mbox{~~for all~~} Y_1, Y_2 \in \stiefel. 
		\end{equation*}
		Thus, we can obtain that
		\begin{equation*}
		\begin{aligned}
			\norm{ (X^{(k+1)})\zz\nabla f(X^{(k+1)}) - (X^{(k+1)})\zz \nabla f(\bar{X}^{(k)}) }\ff 
			\leq {} & \norm{ X^{(k+1)} }_2  \norm{ \nabla f(X^{(k+1)}) - \nabla f(\bar{X}^{(k)}) }\ff \\
			\leq {} & \rho \norm{ X^{(k+1)} - \bar{X}^{(k)} }\ff,
		\end{aligned}
		\end{equation*}
	 	and similarly,
	 	\begin{equation*}
	 		\norm{ \nabla f(\bar{X}^{(k)})\zz {X^{(k+1)}} - \nabla f(X^{(k+1)})\zz X^{(k+1)} }\ff
	 		\leq \rho \norm{ X^{(k+1)} - \bar{X}^{(k)} }\ff.
	 	\end{equation*}
		Together with \eqref{eq:temp-1}, we can conclude that
		\begin{align*}
			& \norm{ (X^{(k+1)})\zz \nabla f(X^{(k+1)}) - \nabla f(X^{(k+1)})\zz X^{(k+1)} }\ff \\
			\leq {} & \norm{ (X^{(k+1)})\zz\nabla f(X^{(k+1)}) - (X^{(k+1)})\zz \nabla f(\bar{X}^{(k)}) }\ff
			+\left\| (X^{(k+1)})\zz \nabla f(\bar{X}^{(k)}) \right. \\
			& \left. - \nabla f(\bar{X}^{(k)})\zz {X^{(k+1)}} \right\|\ff
			+\norm{ \nabla f(\bar{X}^{(k)})\zz {X^{(k+1)}} - \nabla f(X^{(k+1)})\zz X^{(k+1)} }\ff \\
			\leq {} & 2 \dkh{ \rho + \gamma } \norm{ X^{(k+1)} - \bar{X}^{(k)} }\ff,
		\end{align*}
		and complete the proof.
	\end{proof}
	
	Next we show the distance between $X^{(k+1)}$ and $\bar{X}^{(k)}$ converges to 0.
	
	\begin{lemma}
		\label{le:dk+1tilde}
		Suppose Assumption \ref{asp:Gao-1} holds, $\gamma>\rho$, 
		and $\{ X^{(k)} \}$ is the iterate sequence generated by Algorithm~\ref{alg:MCM}.
		Then it holds that
		\begin{equation*}
			\lim\limits_{k\to\infty} \norm{ X^{(k+1)} - \bar{X}^{(k)} }\ff = 0.
		\end{equation*}
	\end{lemma}
	
	\begin{proof}
		Firstly, it follows from the inequality \eqref{eq:des-barx} that
		\begin{equation*}
			8c_{\gamma} \tr \dkh{ \Sigma + U \Sigma V\zz }
			\geq \norm{ (\bar{X}^{(k)})\zz \nabla f(\bar{X}^{(k)}) - \nabla f (\bar{X}^{(k)})\zz\bar{X}^{(k)} }\fs 
			\geq 0.
		\end{equation*}
		Then by simple calculations, we can obtain that
		\begin{align*} 
			& \inner{ X^{(k+1)} - \bar{X}^{(k)} }{ X^{(k+1)} - \bar{X}^{(k)} + 2\gamma\inv \nabla f(\bar{X}^{(k)}) } \\
			= {} & \inner{ X^{(k+1)} - \bar{X}^{(k)} }{  X^{(k+1)} - \bar{X}^{(k)} } 
			+ 2\gamma\inv  \inner{ Q^{(k)} - I_p }{ (\bar{X}^{(k)})\zz \nabla f(\bar{X}^{(k)}) } \\
			= {} & -2\gamma\inv \tr \dkh{ \Sigma + U \Sigma V\zz} 
			\leq 0.
		\end{align*}
		This relationship 
		can guarantee that 
		\begin{equation*}
			\norm{ X^{(k+1)} - \bar{X}^{(k)} + \gamma\inv \nabla f(\bar{X}^{(k)}) }\ff 
			\leq \gamma\inv \norm{\nabla f(\bar{X}^{(k)})}\ff,
		\end{equation*}
		which implies that
		\begin{equation*}
			X^{(k+1)} \in \mathcal{B}\dkh{ 
				\bar{X}^{(k)} - \gamma\inv\nabla f (\bar{X}^{(k)}), \;
				\gamma\inv \norm{\nabla f (\bar{X}^{(k)}) }\ff 
			}.
		\end{equation*}
		We recall  \cite[Lemma 3.1]{Gao2018} and obtain that 
		\begin{equation}
			\label{eq:fb}
			\norm{ X^{(k+1)} - \bar{X}^{(k)} }\fs
			\leq
			\dfrac{2}{\gamma - \rho} \dkh{ f (\bar{X}^{(k)}) - f (X^{(k+1)}) }
			\leq 
			\dfrac{2}{\gamma - \rho} \dkh{ f (X^{(k)}) - f (X^{(k+1)}) }.
		\end{equation}
		Since $\{f(X^{(k)})\}$ is convergent, we conclude that
		\begin{equation*}
			\lim\limits_{k\to\infty} \norm{ X^{(k+1)} - \bar{X}^{(k)} }\ff = 0.
		\end{equation*}
		This completes the proof.
	\end{proof}
	
	Finally, we are ready to present our main convergence result.
	
	\begin{theorem}
		\label{thm:sublinear}
		Suppose Assumption \ref{asp:Gao-1} holds, $\gamma > \rho$,
		and $\{ X^{(k)} \}$ is the iterate sequence generated by Algorithm \ref{alg:MCM}.
		Then there exists at least one convergent subsequence of $\{ X^{(k)} \}$. 
		Furthermore, each accumulation point $X^{\ast}$ of $\{ X^{(k)} \}$ 
		satisfies the first-order stationarity condition \eqref{eq:OC}. 
		More precisely, the following inequality
		\begin{equation*}
			\min\limits_{1\leq k \leq K} \norm{ \nabla f (X^{(k)}) - X^{(k)} \nabla f (X^{(k)})\zz X^{(k)} }\ff 
			\leq \sqrt{\dfrac{c_2 \dkh{ f (X^{(0)}) - \underline{f} } }{K}},
		\end{equation*}
		holds for any $K \geq 1$, 
		where $c_2 > 0$ is a constant defined by
		\begin{equation}
			\label{eq:C2}
			c_2 = \dfrac{1}{c_1} + \dfrac{8 \dkh{ \gamma + \rho }^2}{\gamma - \rho}.
		\end{equation}
	\end{theorem}
	
	\begin{proof}
		It follows from the compactness of the Stiefel manifold $\stiefel$ that $\{X^{(k)}\}$ is bounded, 
		which implies $\{X^{(k)}\}$ has at least one convergent subsequence. 
		Suppose $X^{\ast}$ is an accumulation point of $\{X^{(k)}\}$.
		It is clear that $X^{\ast}\in\stiefel$ due to the feasibility of $\{X^{(k)}\}$. 
		
		Recalling the convergence of $\{f(X^{(k)})\}$ and \eqref{eq:fvdk+1}, we have 
		\begin{equation*}
			\lim\limits_{k\to\infty} \norm{ \dkh{I_n - X^{(k)}(X^{(k)})\zz } \nabla f (X^{(k)}) }\ff = 0,
		\end{equation*}
		which directly implies
		\begin{equation}
			\label{eq:subs}
			\dkh{ I_n - X^{\ast} (X^{\ast})\zz } \nabla f(X^{\ast}) = 0.
		\end{equation}
		On the other hand, 
		it follows from Lemma~\ref{le:symk+1} and Lemma~\ref{le:dk+1tilde} that
		\begin{equation*}
			\lim\limits_{k\to\infty} \norm{ (X^{(k)})\zz \nabla f (X^{(k)}) - \nabla f (X^{(k)})\zz X^{(k)} }\ff 
			\leq 2 \dkh{ \rho + \gamma } \lim\limits_{k\to\infty} \norm{ X^{(k)} - \bar{X}^{(k - 1)} }\ff
			= 0,
		\end{equation*}
		which yields that 
		\begin{equation}
			\label{eq:sysm}
			(X^{\ast})\zz \nabla f (X^{\ast}) = \nabla f (X^{\ast})\zz X^{\ast}.
		\end{equation}
		
		Combining ``feasibility'', ``sub-stationarity" \eqref{eq:subs} and ``symmetry" \eqref{eq:sysm},
		we conclude that $X^*$ satisfies the first-order stationarity condition \eqref{eq:OC}.
		
		Furthermore, it follows from Lemma \ref{le:symk+1} and \eqref{eq:fb} that 
		\begin{equation*}
			\norm{ (X^{(k+1)})\zz \nabla f(X^{(k+1)}) - \nabla f(X^{(k+1)})\zz X^{(k+1)} }\fs 
			\leq \dfrac{8(\gamma+\rho)^2}{\gamma - \rho} \dkh{ f(X^{(k)}) - f(X^{(k+1)}) }.
		\end{equation*}
		Together with the relationships \eqref{eq:prox-first} and \eqref{eq:ssc}, 
		we can arrive at
		\begin{align*}
			&\norm{ \nabla f(X^{(k)}) - X^{(k)} \nabla f (X^{(k)})\zz X^{(k)} }\fs \\
			\leq {} & \dfrac{1}{c_1} \dkh{ f(X^{(k)})-f(X^{(k+1)}) } 
			+ \dfrac{8 \dkh{ \gamma + \rho }^2}{\gamma - \rho} \dkh{ f (X^{(k - 1)}) - f (X^{(k)}) }.
		\end{align*}
		To sum up both sides of the above inequality from $k = 1$ to $K$, we can obtain
		\begin{align*}
			& \sum\limits_{k = 1}^{K} \norm{ \nabla f (X^{(k)}) - X^{(k)} \nabla f (X^{(k)})\zz X^{(k)} }\fs \\
			\leq {} & \dfrac{1}{c_1} \sum\limits_{k = 1}^{K} \dkh{ f (X^{(k)}) - f (X^{(k+1)}) } 
			+ \dfrac{8 \dkh{ \gamma + \rho }^2}{\gamma - \rho} \sum\limits_{k=1}^{K} \dkh{ f (X^{(k - 1)}) - f (X^{(k)}) } \\
			= {} & \dfrac{1}{c_1} \dkh{ f (X^{(1)}) - f (X^{(K + 1)}) } 
			+ \dfrac{8 \dkh{ \gamma + \rho }^2}{\gamma - \rho} \dkh{ f (X^{(0)}) - f (X^{(K)}) } 
			\leq  c_2 \dkh{ f(X^{(0)}) - \underline{f} },
		\end{align*}
		where $c_2$ is defined by \eqref{eq:C2}.
		Together with the fact that
		\begin{equation*}
			\sum\limits_{k=1}^{K} \norm{ \nabla f (X^{(k)}) - X^{(k)}\nabla f(X^{(k)})\zz X^{(k)} }\fs 
			\geq  K \min\limits_{1\leq k \leq K} \norm{ \nabla f (X^{(k)}) - X^{(k)}\nabla f (X^{(k)})\zz X^{(k)} }\fs,
		\end{equation*}
		we complete the proof.
	\end{proof}
	
	\begin{remark}
		According to the stopping criterion,
		Theorem~\ref{thm:sublinear} guarantees the termination of 
		Algorithm~\ref{alg:MCM} in at most $O(1/\epsilon^2)$ iterations.
	\end{remark}

\section{Numerical experiments}

\label{sec:ne}

	In this section, we report the numerical performance of the algorithms based on Algorithm~\ref{alg:MCM}.
	Two types of testing problems are introduced in Subsection~\ref{subsec:tp}.
	The implementation details including the selection of algorithm parameters and stopping criterion 
	are presented in Subsection~\ref{subsec:id}.
	The numerical comparison among our algorithms and those introduced in \cite{Gao2018}
	is presented in Subsection~\ref{subsec:com}.
	Finally, we compare our algorithms with other two state-of-the-art approaches,
	and numerical results are shown in  Subsection~\ref{subsec:pc}.
	All experiments are performed on a workstation 
	with one Intel(R) Xeon(R) Silver 4110 CPU (at 2.10GHz$\times$32) 
	and 384GB of RAM running in MATLAB R2018a under Ubuntu 18.10.

\subsection{Testing problems}

\label{subsec:tp}

	{\bf Problem 1.} The first class of testing problems is a quadratic objective minimization over
	the Stiefel manifold:
	\begin{equation*}
		\begin{aligned}
			\min\limits_{X\in \Rnp} \hspace{2mm} & f_1(X) = \dfrac{1}{2}\tr \dkh{ X\zz M X } + \tr \dkh{ N\zz X } \\
			\st \hspace{4mm} & X\zz X = I_p.
		\end{aligned}
	\end{equation*}
	In the experiments, $M \in \Rnn$ and $N \in \Rnp$ are randomly generated by 
	\begin{equation*}
		M = E \varPsi E\zz, \quad N = \alpha Q D,
	\end{equation*}
	where $E = \qr{ \randn{n}{n} } \in \Rnn$, 
	$\tilde{Q} = \randn{n}{p} \in \Rnp$, 
	and $Q \in \Rnp$ with $Q_i = \tilde{Q}_i / \norm{ \tilde{Q}_i }_2 (i = 1, \dotsc, p)$. 
	The notation $\randn{n}{m}$ represents an $n \times m$ matrix 
	randomly generated by i.i.d. standard Gaussian distribution.
	Moreover, $\varPsi \in \Rnn$ and $D \in \Rpp$ are diagonal matrices with, respectively,
	\begin{equation*}
		\varPsi_{ii} = 
		\left\{
		\begin{array}{ll}
			\eta^{1 - i}, & \mbox{if~} \omega_i < 0.5, \\
			-\eta^{1-i}, &  \mbox{otherwise},
		\end{array}
		\right.
		\mbox{~~for all~~} i = 1, 2, \dotsc, n,
	\end{equation*}
	\begin{equation*}
		D_{ii} = \zeta^{1 - i}, \mbox{~~for all~~} i = 1, 2, \dotsc, p,
	\end{equation*}
	where $\omega_i \in [0, 1]$ for $i = 1,2,\dots,n$ are randomly generated numbers. 
	Here, $\eta \geq 1$ is a parameter determining the decay of eigenvalues of $M$, 
	and $\zeta \geq 1$ is a parameter referring to the growth rate of the column’s norm of $N$.
	The parameter $\alpha > 0$ represents the scale difference between the quadratic term and the linear term. 
	Unless otherwise stated, the default values of these parameters are $\eta = 1.01$, $\zeta = 1.01$, $\alpha = 1$. 
	This class of testing problems is also used in \cite{Gao2018}, which satisfies Assumption~\ref{asp:Gao-2}. 
	
	{\bf Problem 2.} The second class of testing problems is Brockett function minimization over the Stiefel manifold:
	\begin{equation*}
		\begin{aligned}
			\min\limits_{X \in \Rnp} \hspace{2mm} & f_2(X) = \dfrac{1}{2}\tr \dkh{ DX\zz AX} \\
			\st \hspace{4mm} & X\zz X = I_p.
		\end{aligned}
	\end{equation*}
	The data matrix $A \in \Rnn$ is randomly generated by 
	\begin{equation*}
		A = E\varPsi E\zz.
	\end{equation*}
	Here, $E = \qr{ \randn{n}{n} } \in \Rnn$,  $\varPsi \in \Rnn$ and $D \in \Rpp$ are diagonal matrices with, respectively,
	\begin{equation*}
		\varPsi_{ii} =
		\left\{
		\begin{array}{ll}
			\eta^{1-i} + \beta, & \mbox{if~} \omega_i < 0.5, \\
			-\eta^{1-i} - \beta, & \mbox{otherwise},
		\end{array}
		\right.
		\mbox{~~for all~~} i = 1, 2, \dotsc, n,
	\end{equation*}
	\begin{equation*}
		D_{ii} = 
		\left\{
		\begin{array}{ll}
			\alpha\zeta^{1 - i}, & \mbox{if~} \theta_i < 0.5, \\
			-\alpha\zeta^{1 - i}, & \mbox{otherwise},
		\end{array}
		\right.
		\mbox{~~for all~~} i = 1, 2, \dotsc, p,
	\end{equation*}
	where $\omega_i \in [0, 1]$ for $i = 1, 2, \dotsc, n$ 
	and $\theta_i \in [0, 1]$ for $i = 1, 2, \dotsc, p$ 
	are randomly generated numbers. 
	Two parameters $\eta \geq 1$ and $\beta\geq 1$ determine the difference of eigenvalues of $A$. 
	Moreover, $\zeta \geq 1$ is a parameter referring to the decrease rate of diagonal entries of $D$. 
	The parameter $\alpha > 0$ represents the scale difference between $A$ and $D$. 
	Unless otherwise stated, the default values of these parameters are $\eta = 1.05$, $\zeta = 1.05$, $\beta = 2$, $\alpha = 0.1$.
	This class of testing  problems does not satisfy Assumption~\ref{asp:Gao-2}.

\subsection{Implementation details}

\label{subsec:id}

	All of the three algorithms GRP, GPP and CBCDP have a common parameter $\gamma$.
	Although in the theoretical analysis, $\gamma$ should be larger than the constant $\rho$, 
	we set $\gamma = 10^{-3}s$ in practice, 
	where $s$ is an estimation of $\norm{ \nabla^2 f (0) }_2$. 
	More specifically, we choose $s = \norm{M}_2$  and $s = \norm{A}_2 \norm{D}_2$ 
	for Problems 1 and 2, respectively.
	
	In practice, we recommend to use the following alternating BB stepsize introduced in \cite{Dai2005}:
	\begin{equation*}
		\tau^{(k)}_{\mathrm{ABB}}=
		\left\{
		\begin{array}{ll}
			\tau^{(k)}_{\mathrm{BB1}} & \mbox{if~} k \mbox{~is odd}, \\
			\tau^{(k)}_{\mathrm{BB2}} & \mbox{if~} k \mbox{~is even}.
		\end{array}
		\right.
	\end{equation*}
	Here, two Barzilai-Borwein (BB) stepsizes were first introduced in \cite{Barzilai1988}:
	\begin{equation*}
		\tau^{(k)}_{\mathrm{BB1}} = \dfrac{\abs{\inner{ J_{k}}{ K_{k} } } }{\inner{K_{k}}{ K_{k} } }, 
		\text{~~or~~} 
		\tau^{(k)}_{\mathrm{BB2}} = \dfrac{\inner{J_{k}}{ J_{k} } }{\abs{\inner{J_{k}}{ K_{k} } } },
	\end{equation*}
	where $J_{k} = X^{(k)} - X^{(k - 1)}$, $K_{k} = c(X^{(k)}) - c(X^{(k - 1)})$.
	
	As for the CBCDP method, the subproblem~\eqref{eq:CBCD} can be solved globally if our testing problems are quadratic, 
	which has been elaborately introduced in \cite{Gao2018}	and hence omitted here. 
	For the updating order of the block coordinate descent scheme, we simply choose the Gauss--Seidel manner.
	
	The stopping criterion can be described as follows,
	\begin{equation}
	\label{eq:scg}
		\norm{ \nabla f (X^{(k)}) - X^{(k)} \nabla f (X^{(k)})\zz X^{(k)} }\ff 
		\leq \epsilon_g \norm{ \nabla f (X^{(0)}) - X^{(0)} \nabla f (X^{(0)})\zz X^{(0)} }\ff,
	\end{equation}
	where $\epsilon_g > 0$ is a tolerance constant. 
	In addition, we also adopt the following stopping rules based on the relative error:
	\begin{equation}
		\mathrm{tol}^{(k)}_x = \dfrac{\norm{ X^{(k)} - X^{(k - 1)} }\ff}{\sqrt{n}} \leq \epsilon_x, 
		\quad 
		\mathrm{tol}^{(k)}_f = \dfrac{\abs{ f (X^{(k)}) - f (X^{(k - 1)})}}{\abs{f (X^{(k - 1)})} + 1} \leq \epsilon_f, 
		\label{eq:scxf}
	\end{equation}
	and
	\begin{equation}
	\label{eq:sc10xf}
		\mathrm{mean}(\mathrm{tol}^{(k - \min\{k, T\} + 1)}_x, \dotsc, \mathrm{tol}^{(k)}_x) \leq 10\epsilon_x,
		\quad 
		\mathrm{mean}(\mathrm{tol}^{(k - \min\{k, T\} + 1)}_f, \dotsc, \mathrm{tol}^{(k)}_f) \leq 10\epsilon_f,
	\end{equation}
	where $\epsilon_x > 0$ and $\epsilon_f > 0$ are also tolerance constants, 
	and $\mathrm{mean}(a_1, \dotsc, a_m)$ denotes the mean value of numbers $a_1, \dotsc, a_m$.
	We terminate the algorithm when it satisfies one of the above three stopping criteria~\eqref{eq:scg}-\eqref{eq:sc10xf}, 
	or reaches a preset maximum iteration number $\mathrm{MaxIter}$.
	Unless otherwise stated, we set the tolerance parameters $\epsilon_x = 10^{-6}$,
	$T = 5$ and $\mathrm{MaxIter} = 3000$. 
	For Problems 1 and 2, we set $\epsilon_g = 10^{-5}$, $\epsilon_f = 10^{-10}$ 
	and $\epsilon_g = 10^{-3}$, $\epsilon_f = 10^{-8}$, respectively.
	
	In Algorithm \ref{alg:MCM}, the proximal correction step is performed once in each iteration. 
	A special test on GPP employed in solving Problem~2 with $n = 5000$ and $p = 50$
	demonstrates that the decay rate of the ``symmetry'' violation is worse than 
	that of the ``sub-stationarity'' violation. 
	Such unbalance affects the overall performance of our algorithms. 
	Hence, we consider multiple proximal correction steps in each iteration. 
	From Figure~\ref{fig:sym}, we can learn that three times proximal correction 
	can accelerate the decay of ``symmetry'' violation. 
	Heuristically, we recommend $\delta_k = 2\lceil \sqrt{k} / 2 \rceil - 1$ times 
	proximal correction steps in the $k$-th iteration, 
	which substantially makes the two decay rates close to each other. 
	Therefore, in the following comparison, we use $\delta_k$ as the default number
	of proximal correction steps in each iteration.
	
	\begin{figure}[ht]
		\centering
		\subfloat[Single proximal correction step]{
			\label{subfig:sym_1}
			\begin{minipage}[t]{0.32\linewidth}
				\centering
				\includegraphics[width=1\textwidth]{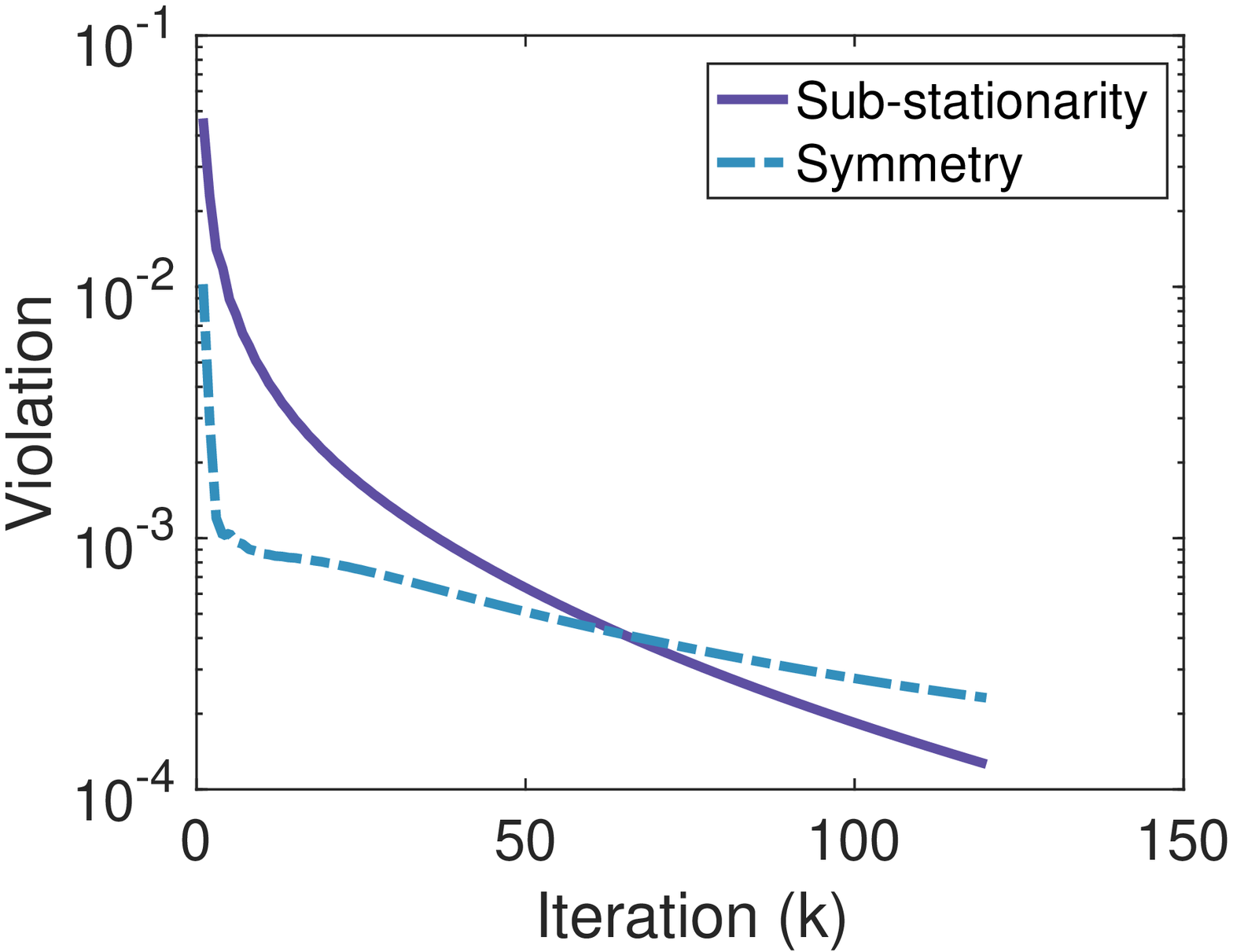}
			\end{minipage}
		}
		\subfloat[3 proximal correction steps ]{
			\label{subfig:sym_3}
			\begin{minipage}[t]{0.32\linewidth}
				\centering
				\includegraphics[width=1\textwidth]{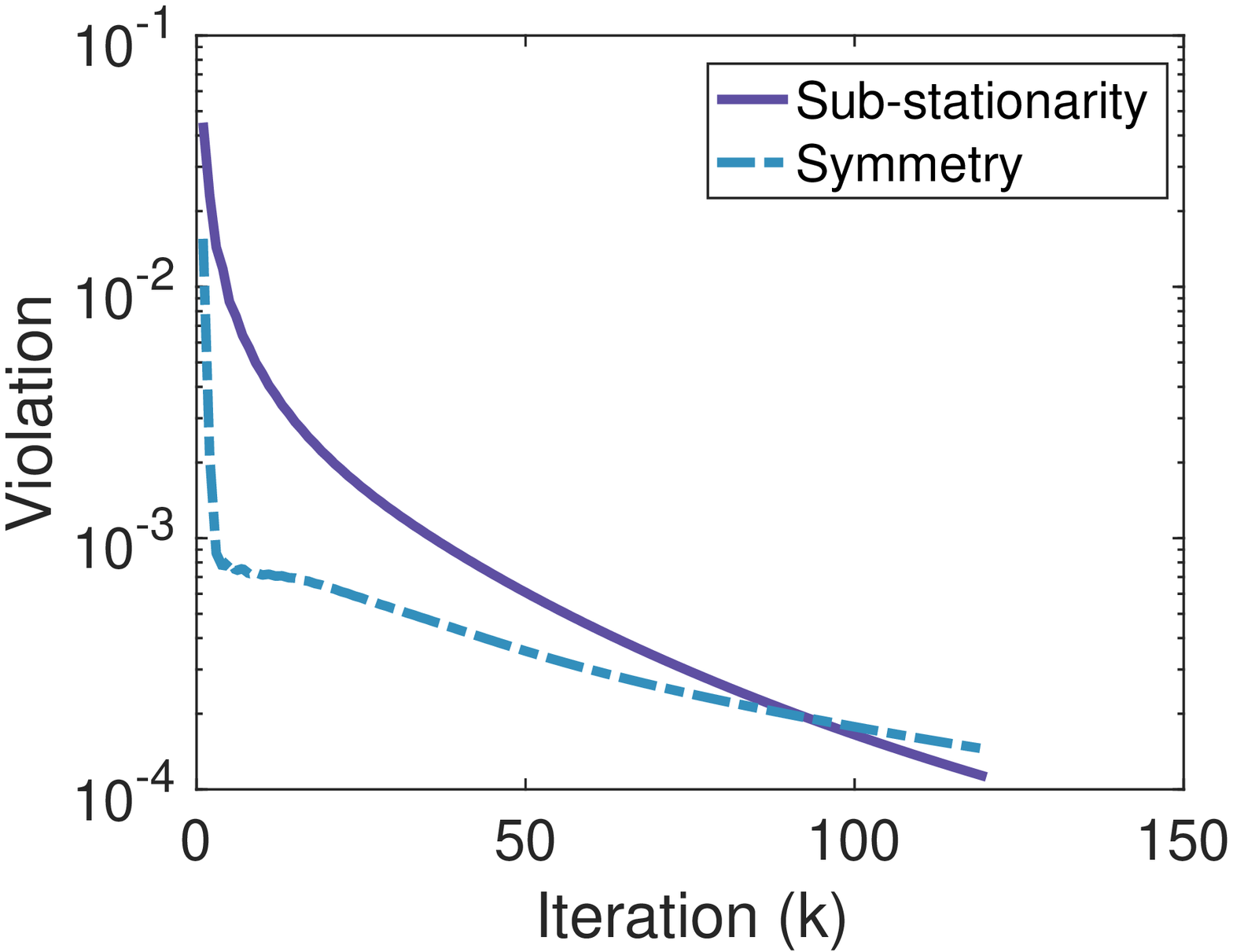}
			\end{minipage}
		}	
		\subfloat[$\delta_k$ proximal correction steps ]{
			\label{subfig:sym_0}
			\begin{minipage}[t]{0.32\linewidth}
				\centering
				\includegraphics[width=1\textwidth]{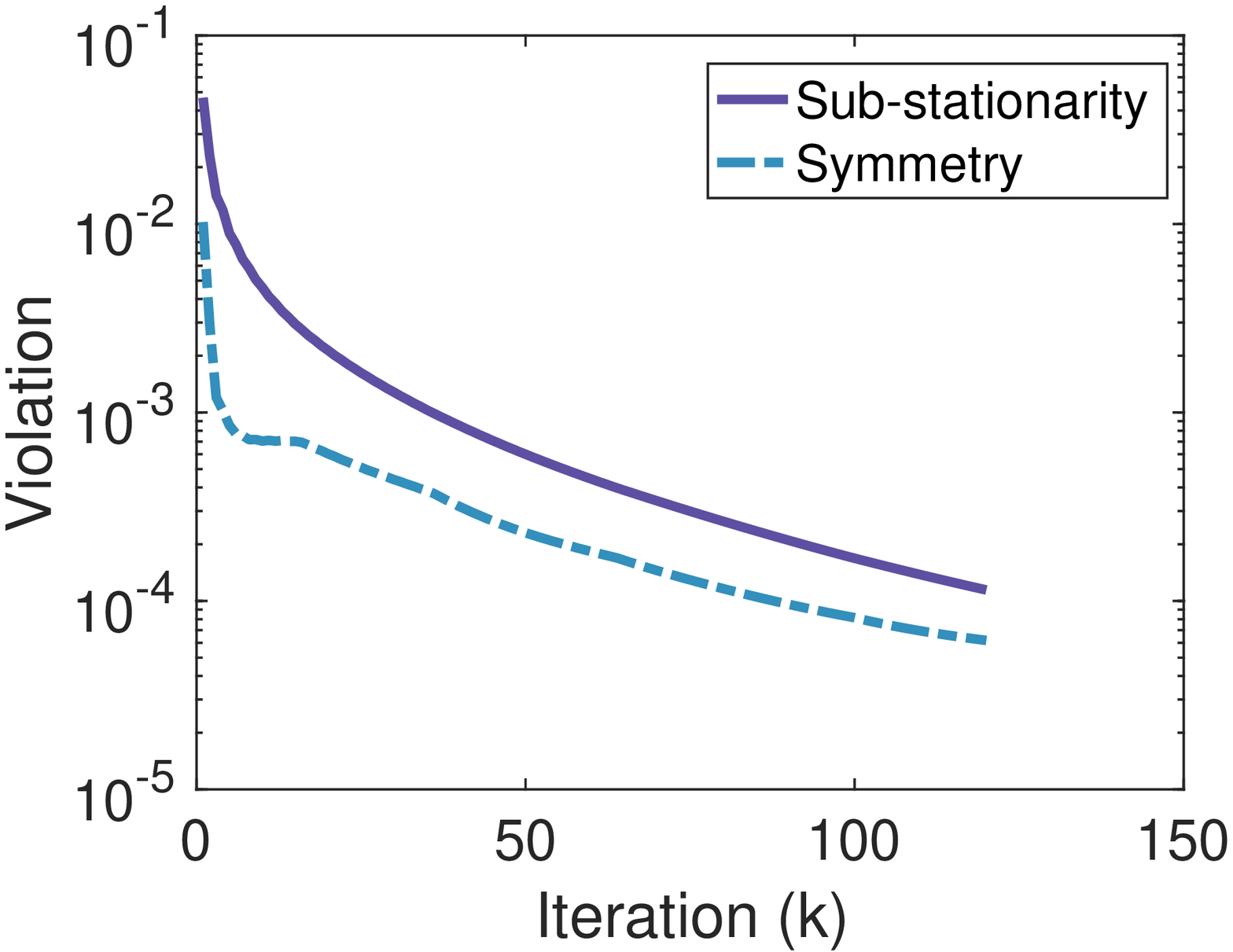}
			\end{minipage}
		}
		\caption{Comparison of  multiple proximal correction steps.}
		\label{fig:sym}
	\end{figure}
	
	We use three measurements in the numerical comparison, including CPU time in seconds, 
	KKT violation ($\norm{ \nabla f (X) - X\nabla f (X)\zz X }\ff$) 
	and function value variance, which is defined as 
	$\abs{ f_s - f_{\min} } / \dkh{ 1 + \abs{ f_{\min} } } + \mathrm{eps}$.
	Here, $f_s$ and $f_{\min}$ refer to the final objective function value returned by solver $s$ 
	and the smallest one of those obtained by all solvers in the comparison, respectively.
	We add $\mathrm{eps} = 2.2204 \times 10^{-16}$, the machine precision in MATLAB, 
	to the relative variance of function value for the sake of logarithmic scale demonstration. 
	Finally, all the tested algorithms are initiated from the same point $X^{(0)}$,
	which is randomly generated by $X^{(0)} = \qr{\randn{n}{p}} \in \stiefel$.

\subsection{Comparison with GR, GP and CBCD}

\label{subsec:com}

	In this subsection, we mainly compare our GRP, GPP and CBCDP with GR, GP, and CBCD, respectively.
	In the test,  all of GR, GP, and CBCD are taken their default settings introduced in \cite{Gao2018}, 
	which are almost the same as our algorithms,
	except for completely different multipliers correction step.
	
	\begin{figure}[ht!]
		\centering
		\subfloat[CPU time (s)]{
			\label{subfig:GR_time}
			\begin{minipage}[t]{0.32\linewidth}
				\centering
				\includegraphics[width=1\textwidth]{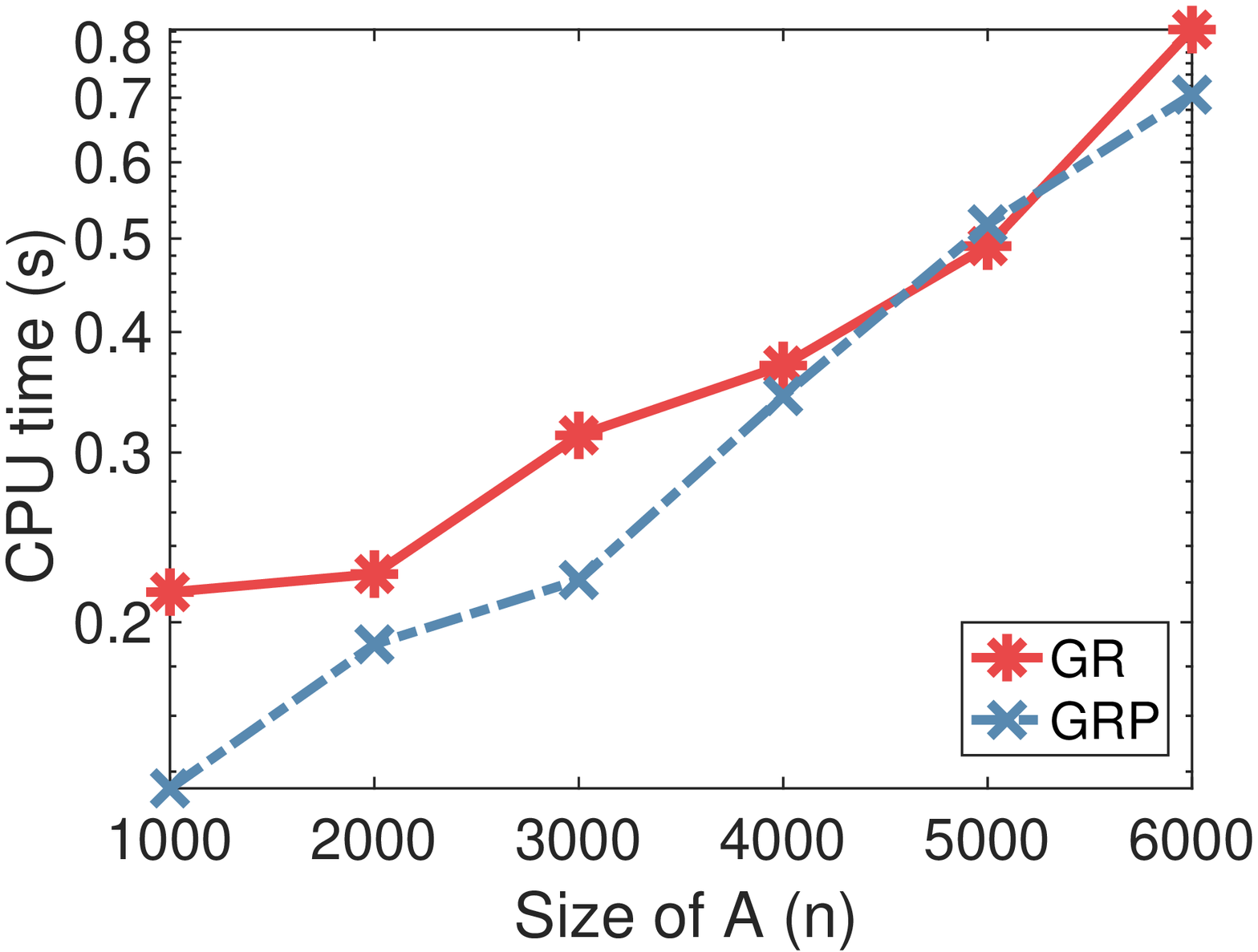}
			\end{minipage}
		}
		\subfloat[Function value variance]{
			\label{subfig:GR_fval}
			\begin{minipage}[t]{0.32\linewidth}
				\centering
				\includegraphics[width=1\textwidth]{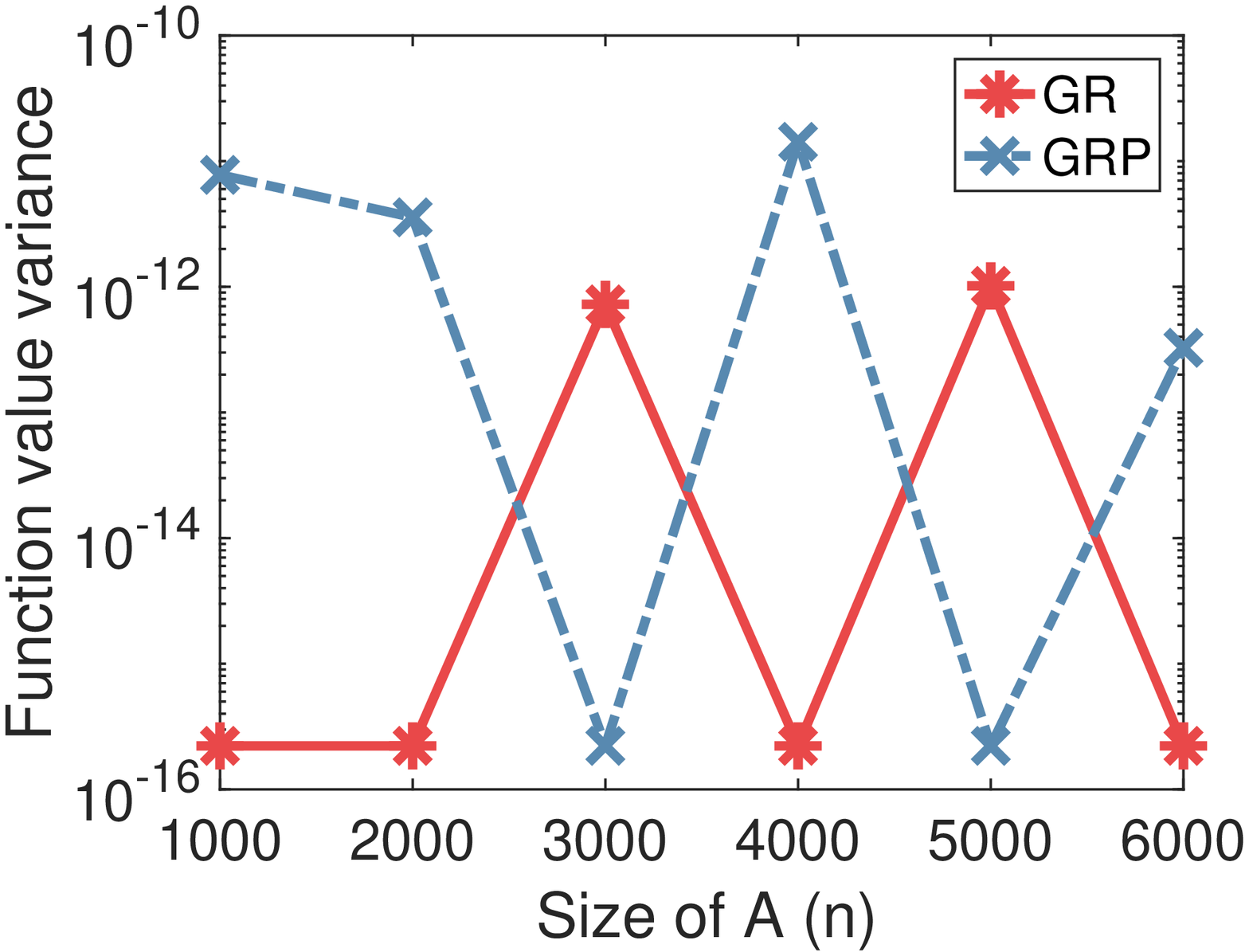}
			\end{minipage}
		}	
		\subfloat[KKT violation]{
			\label{subfig:GR_kkt}
			\begin{minipage}[t]{0.32\linewidth}
				\centering
				\includegraphics[width=1\textwidth]{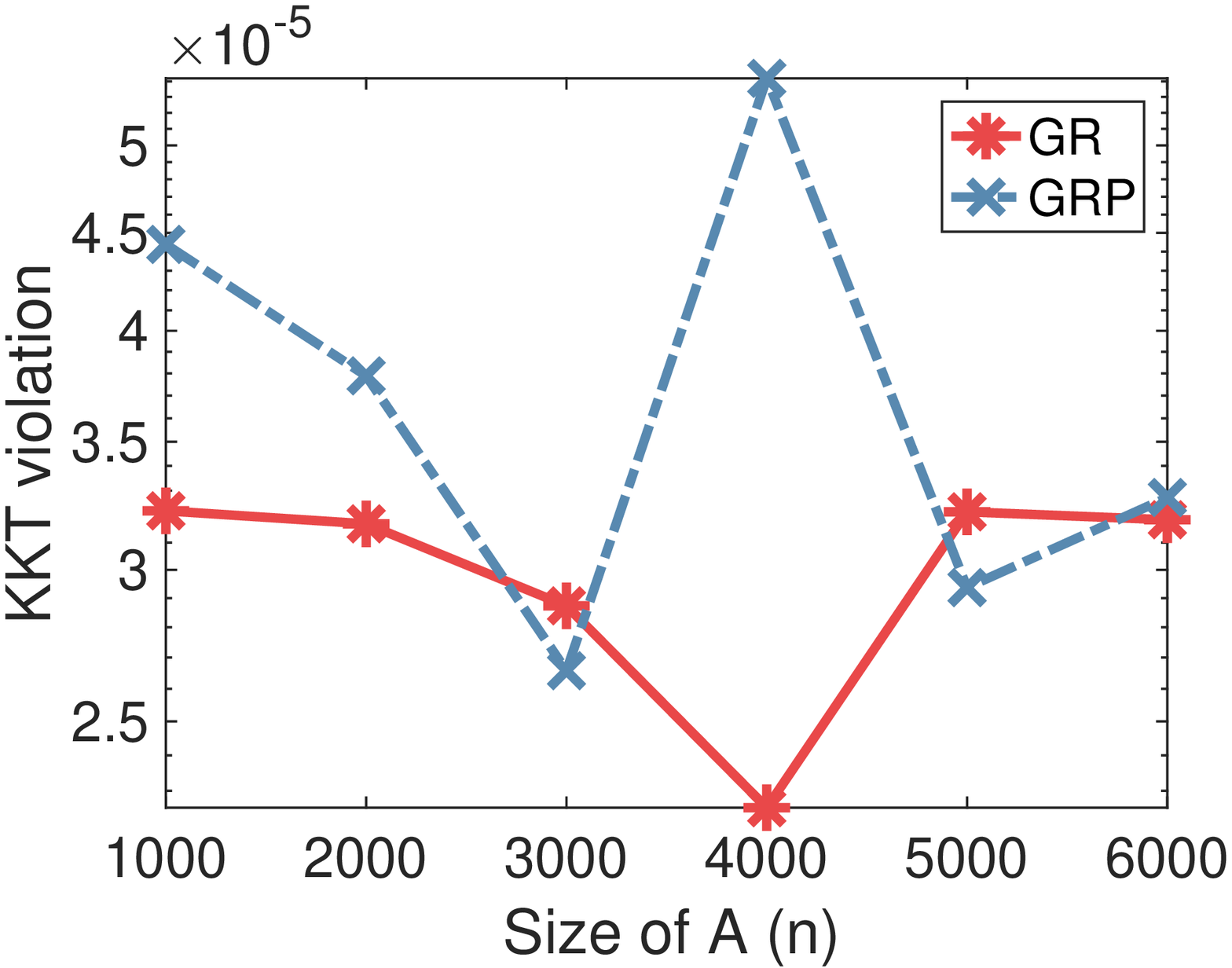}
			\end{minipage}
		}
		
		\subfloat[CPU time (s)]{
			\label{subfig:GP_time}
			\begin{minipage}[t]{0.32\linewidth}
				\centering
				\includegraphics[width=1\textwidth]{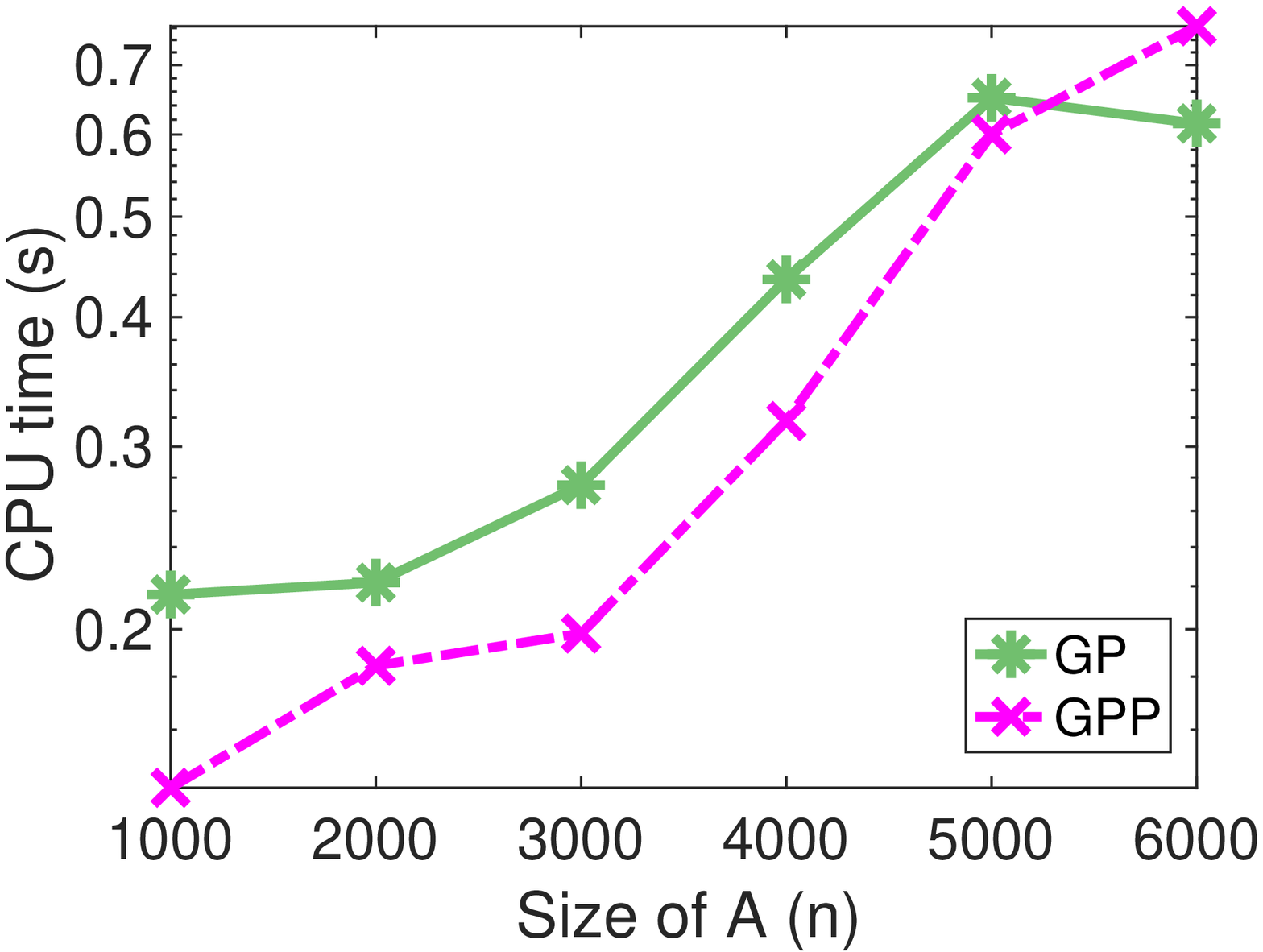}
			\end{minipage}
		}
		\subfloat[Function value variance]{
			\label{subfig:GP_fval}
			\begin{minipage}[t]{0.32\linewidth}
				\centering
				\includegraphics[width=1\textwidth]{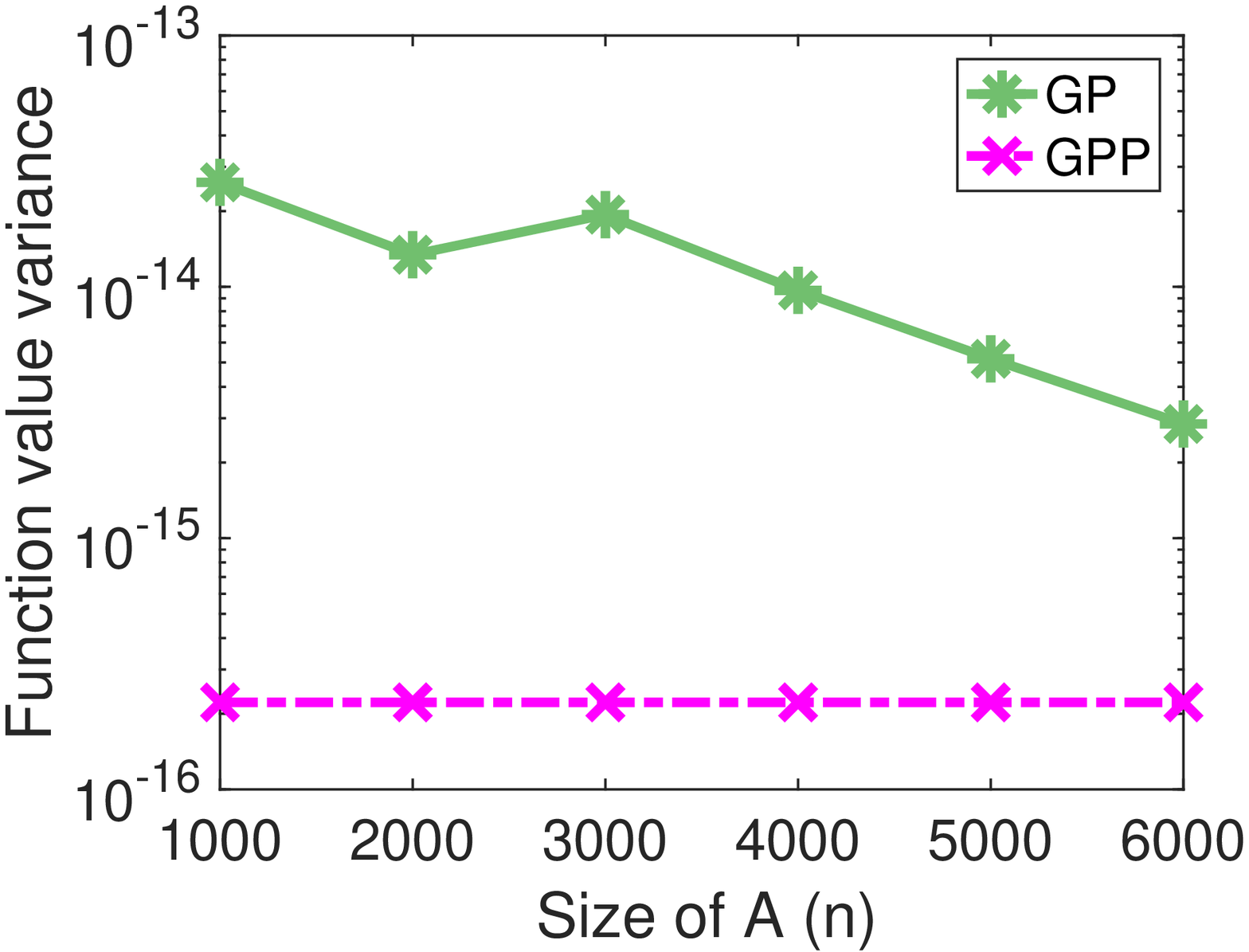}
			\end{minipage}
		}	
		\subfloat[KKT violation]{
			\label{subfig:GP_kkt}
			\begin{minipage}[t]{0.32\linewidth}
				\centering
				\includegraphics[width=1\textwidth]{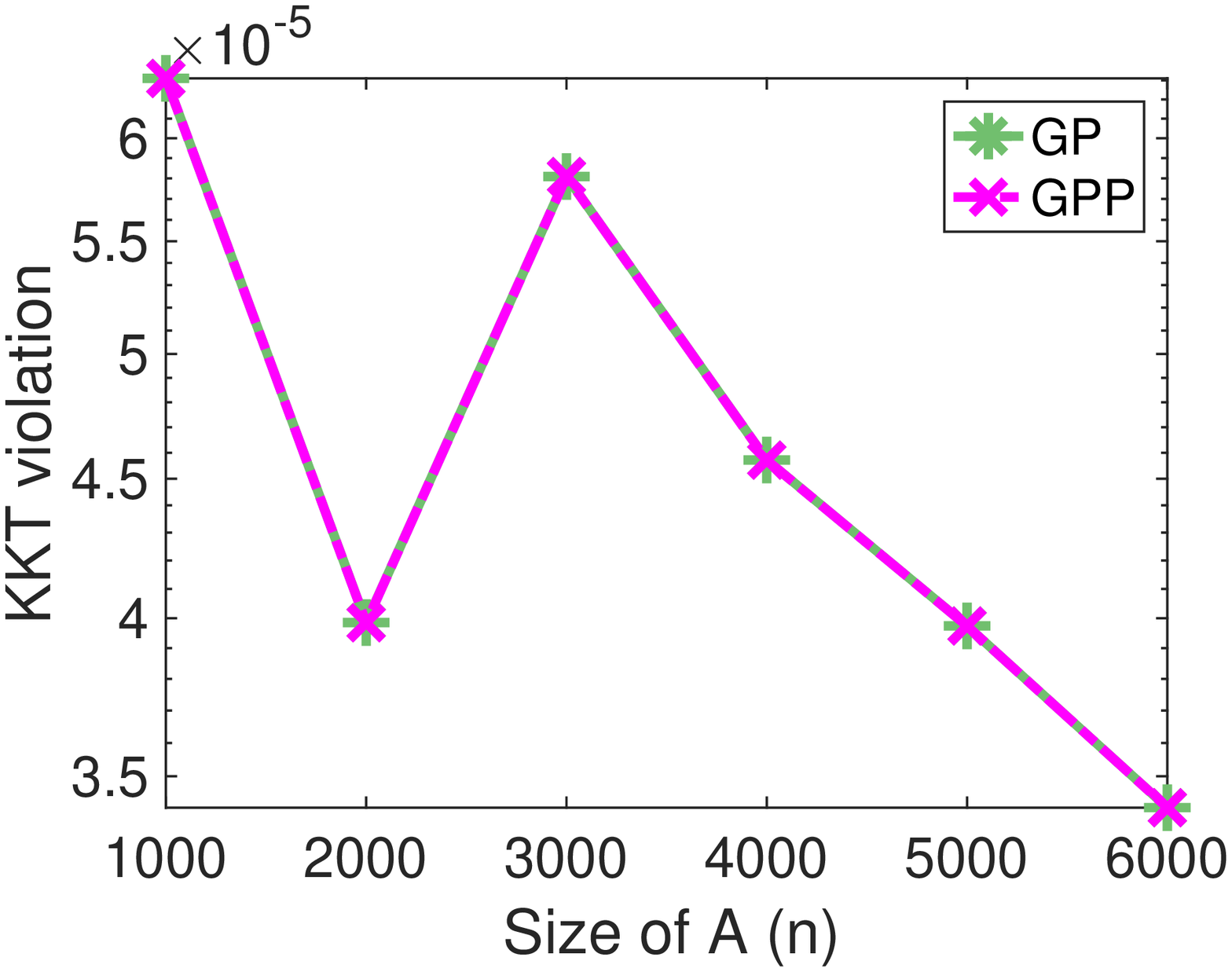}
			\end{minipage}
		}
		
		\subfloat[CPU time (s)]{
			\label{subfig:CBCD_time}
			\begin{minipage}[t]{0.32\linewidth}
				\centering
				\includegraphics[width=1\textwidth]{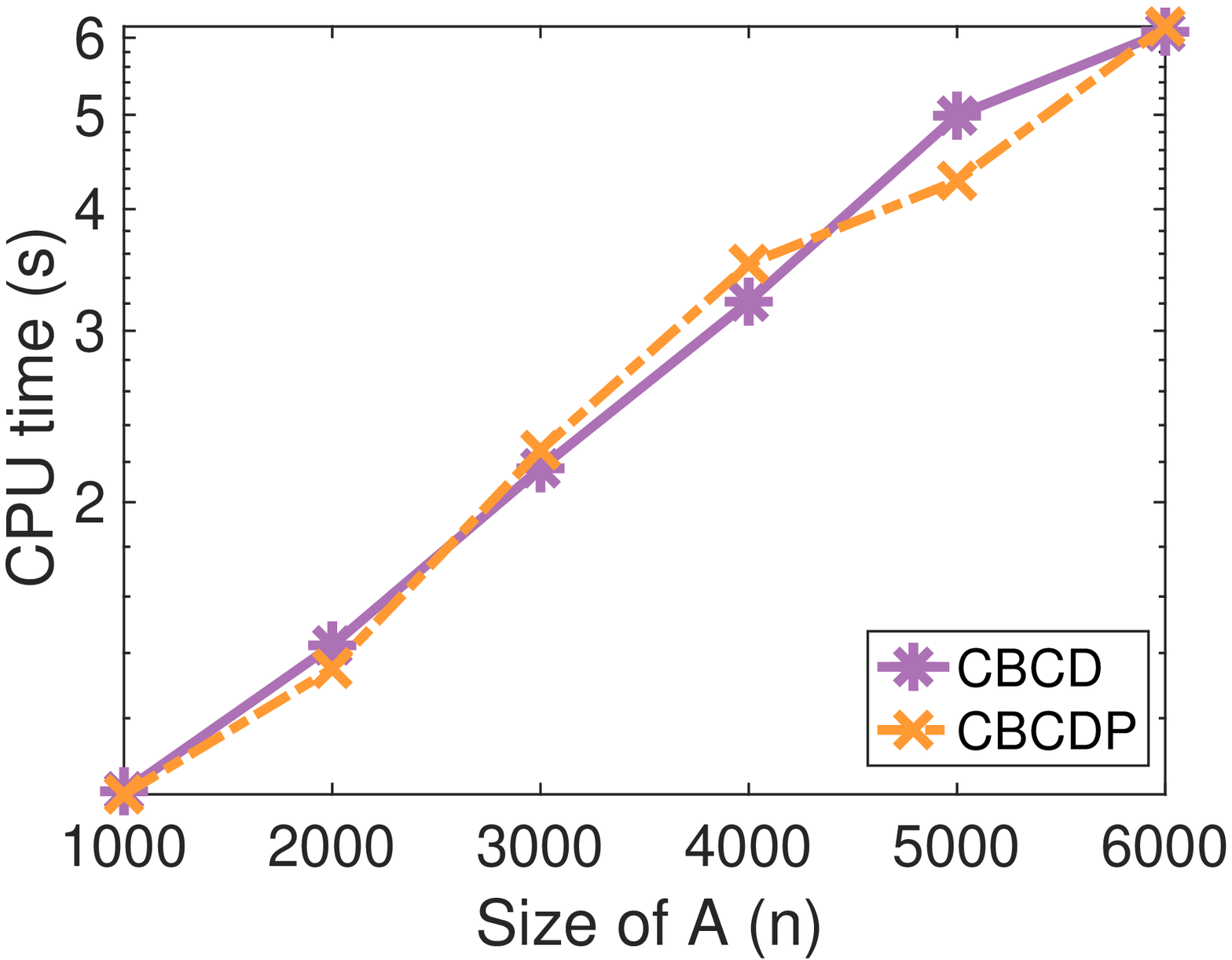}
			\end{minipage}
		}
		\subfloat[Function value variance]{
			\label{subfig:CBCD_fval}
			\begin{minipage}[t]{0.32\linewidth}
				\centering
				\includegraphics[width=1\textwidth]{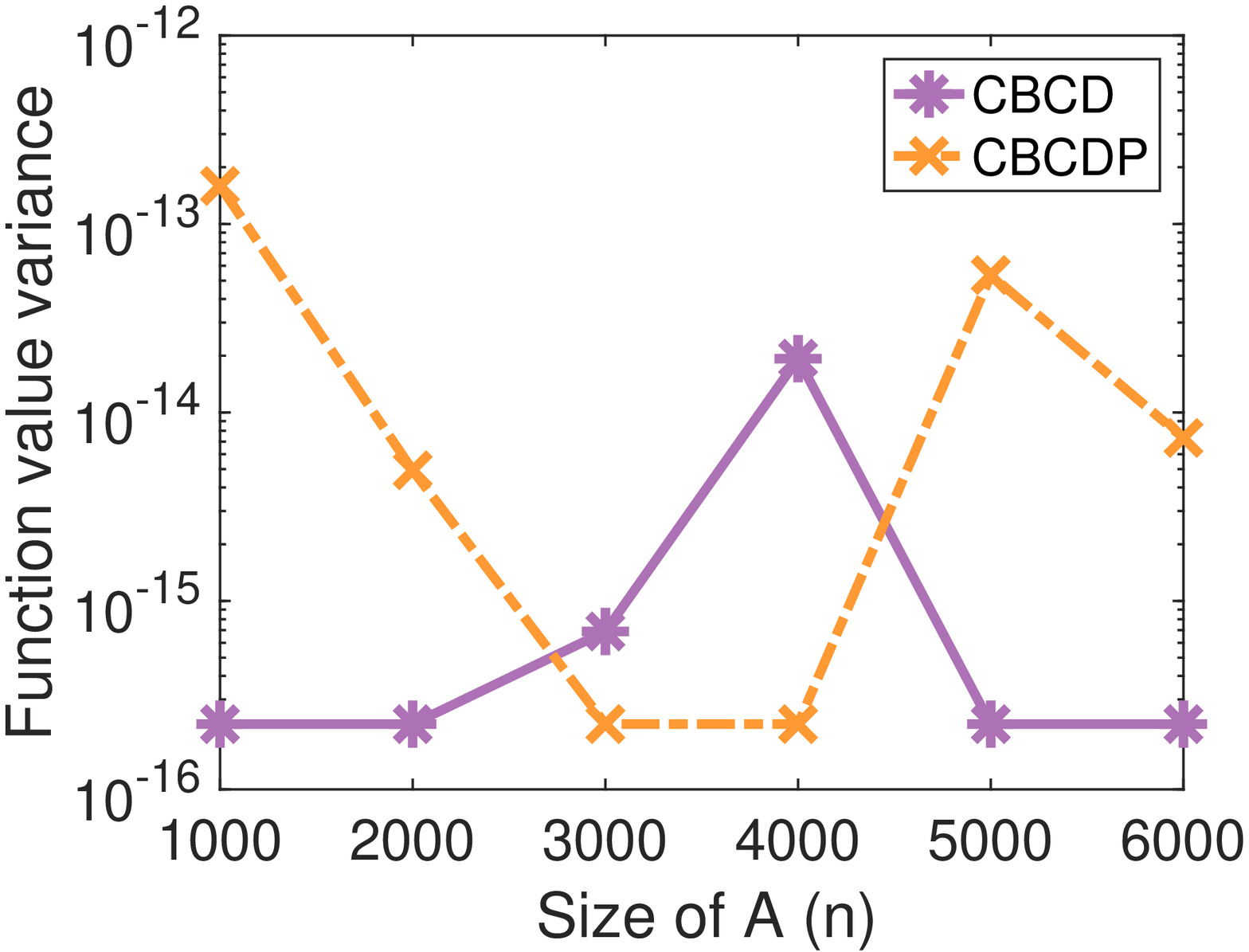}
			\end{minipage}
		}	
		\subfloat[KKT violation]{
			\label{subfig:CBCD_kkt}
			\begin{minipage}[t]{0.32\linewidth}
				\centering
				\includegraphics[width=1\textwidth]{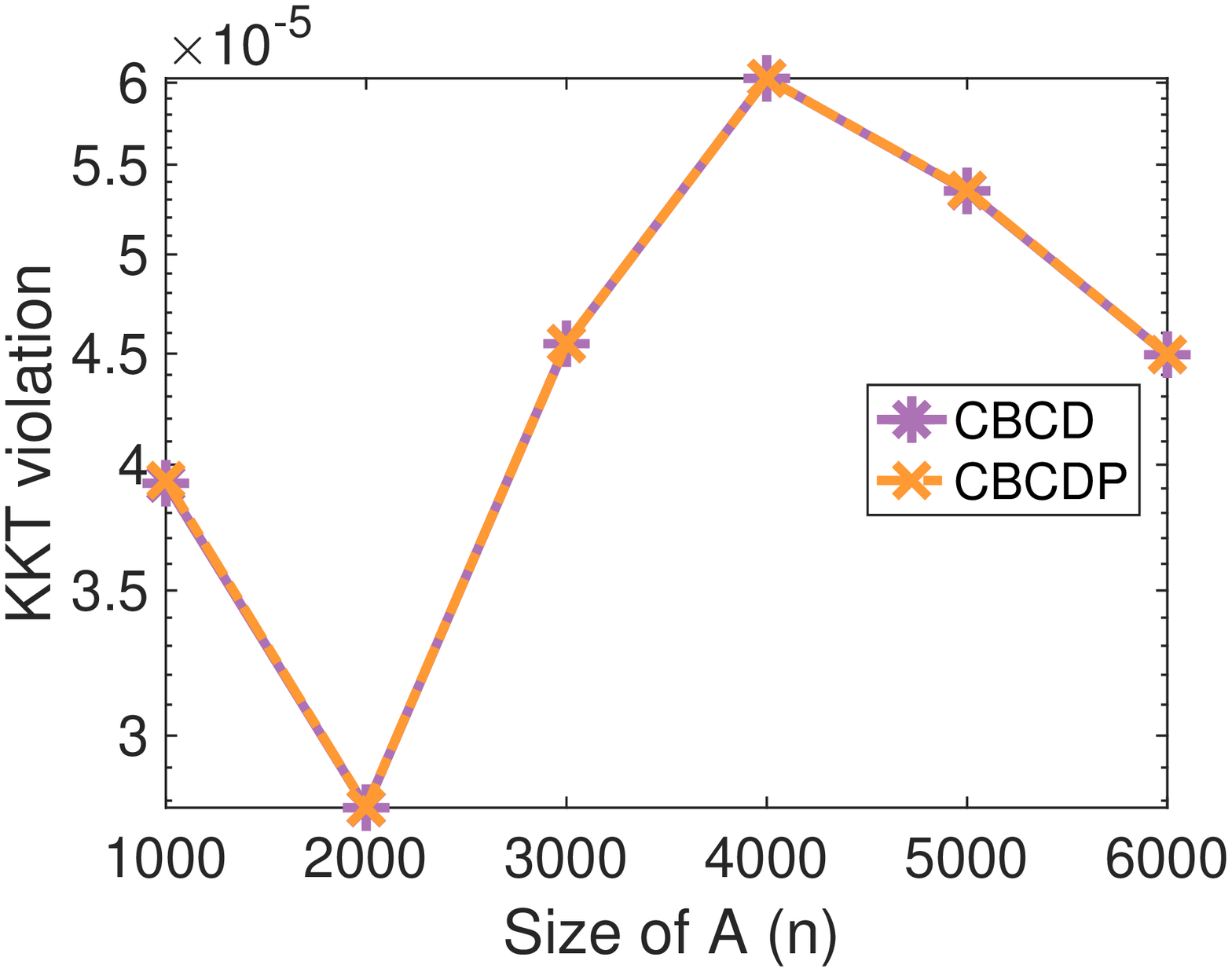}
			\end{minipage}
		}
		
		\caption{Comparison between multipliers correction methods with their original versions.}
		\label{fig:com}
	\end{figure}
	
	For this purpose, we perform on a set of problems based on Problem~1 
	with $n$ ranging from $1000$ to $6000$ increment $1000$ and $p = 60$. 
	Other parameters take their default values. 
	We demonstrate the numerical results in Figure~\ref{fig:com}.
	We observe that these six algorithms all reach comparable KKT violations and final function values.  
	In most cases, GRP, GPP, and CBCDP require less CPU time than GR, GP, and CBCD, respectively. 
	In this sense, our multipliers correction methods 
	are comparable with those proposed in \cite{Gao2018} 
	for the problems satisfying Assumption~\ref{asp:Gao-2}. 

	We also make a comparison among GRP, GPP and CBCDP,
	when they are employed to solve Problem~1 and Problem~2. 
	In this test, we set $n = 3000$ and $p$ ranging from $20$ to $120$ increment $20$. 
	Other parameters take their default values. 
	The numerical results are illustrated in Figure~\ref{fig:width}.
	We observe that GPP outperforms GRP and CBCDP in most cases. 
	Therefore, we choose GPP to represent our new multipliers correction methods 
	in the following numerical experiments.
	
	\begin{figure}[ht!]
		\centering
		\subfloat[CPU time (s) (Problem 1)]{
			\label{subfig:width_1_time}
			\begin{minipage}[t]{0.32\linewidth}
				\centering
				\includegraphics[width=1\textwidth]{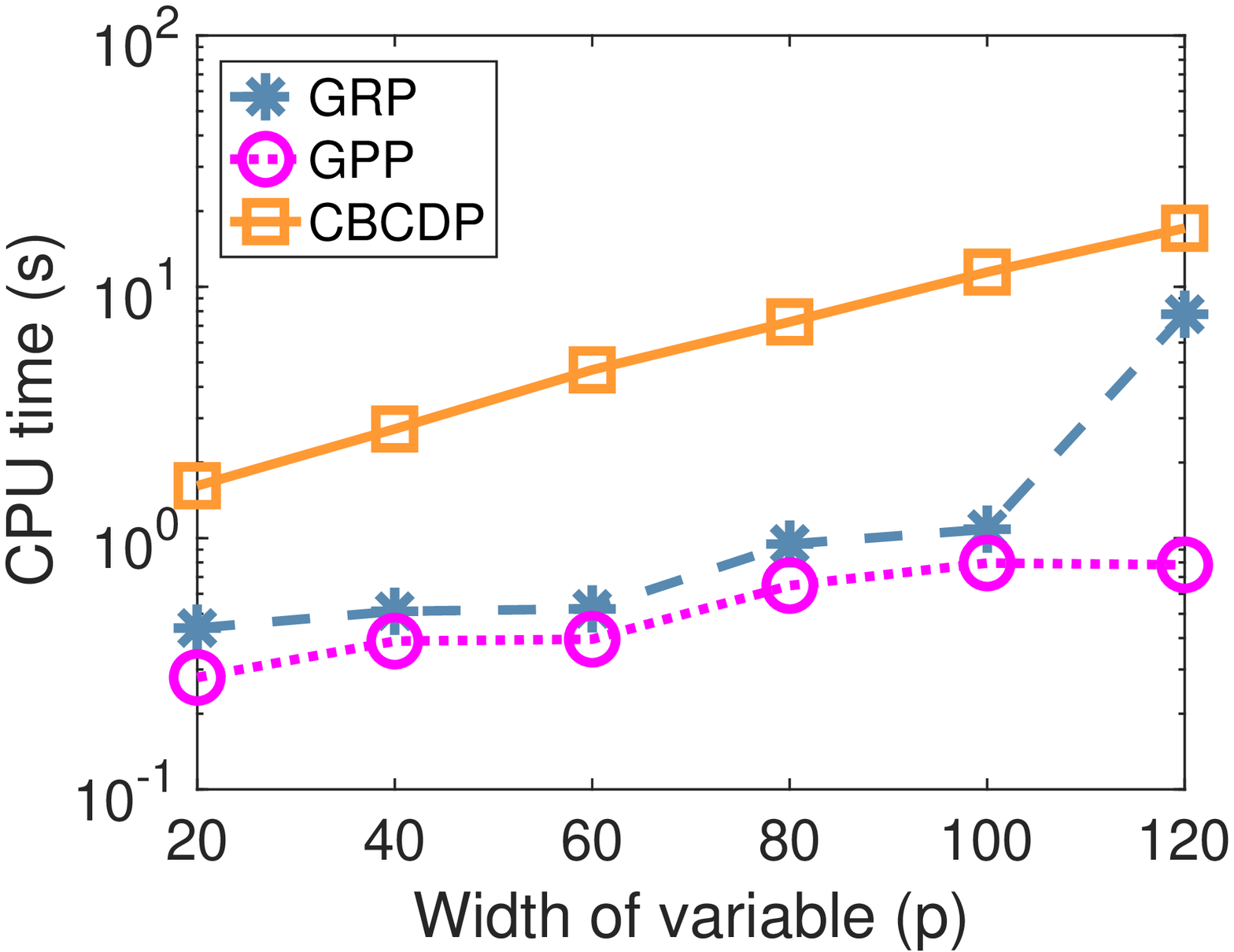}
			\end{minipage}
		}
		\subfloat[Function value variance (Problem 1)]{
			\label{subfig:width_1_fval}
			\begin{minipage}[t]{0.32\linewidth}
				\centering
				\includegraphics[width=1\textwidth]{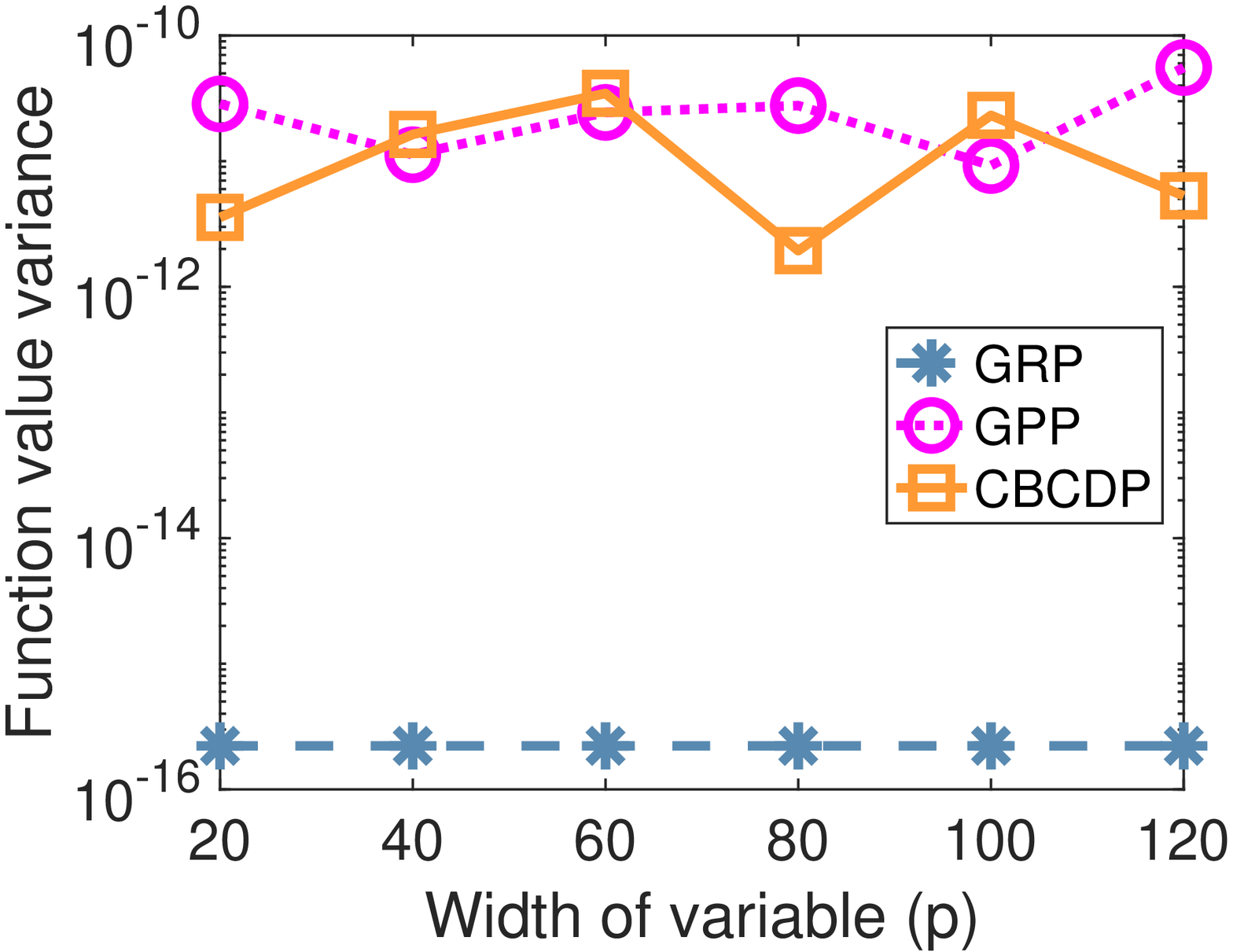}
			\end{minipage}
		}	
		\subfloat[KKT violation (Problem 1)]{
			\label{subfig:width_1_kkt}
			\begin{minipage}[t]{0.32\linewidth}
				\centering
				\includegraphics[width=1\textwidth]{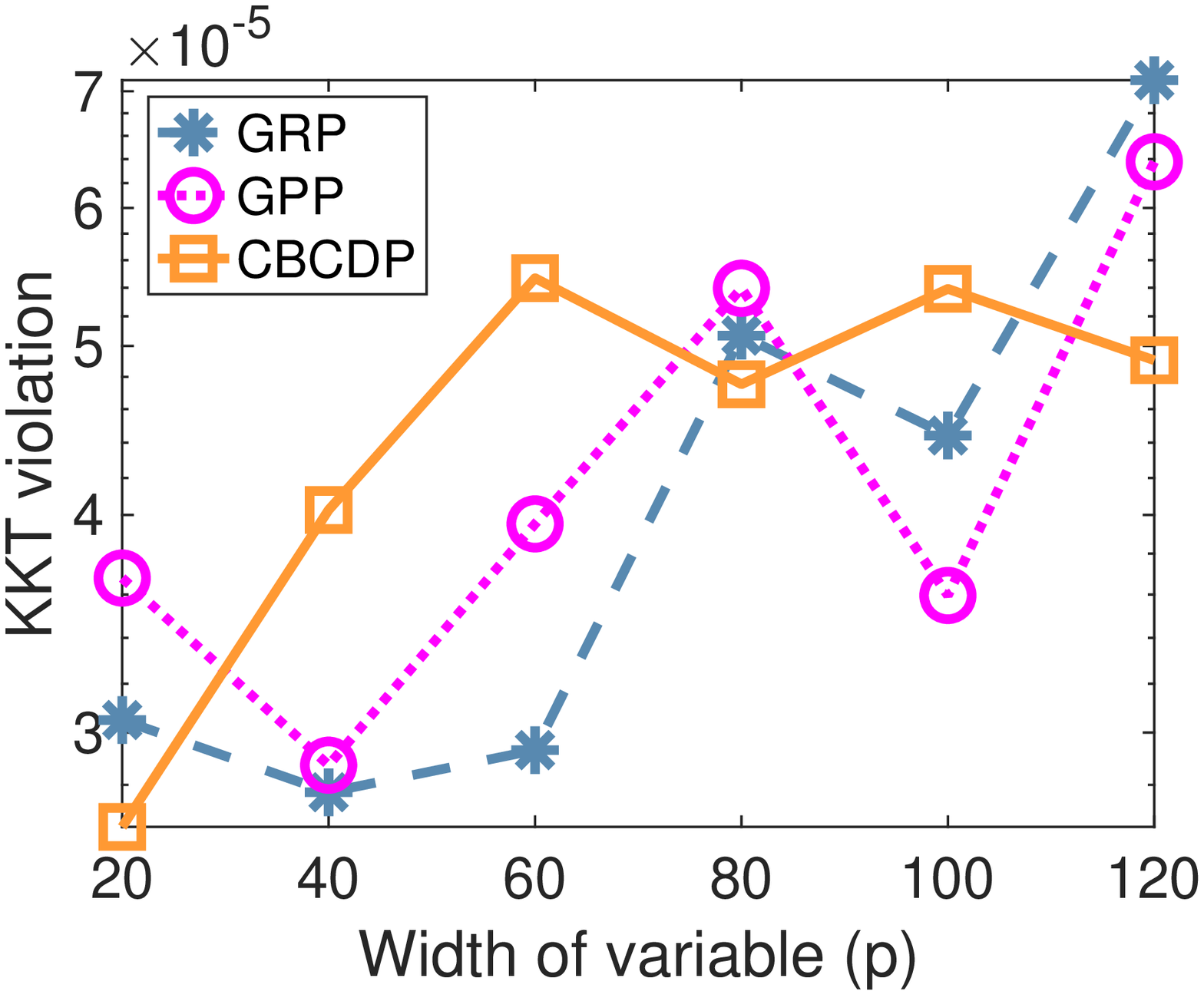}
			\end{minipage}
		}
		
		\subfloat[CPU time (s) (Problem 2)]{
			\label{subfig:width_2_time}
			\begin{minipage}[t]{0.32\linewidth}
				\centering
				\includegraphics[width=1\textwidth]{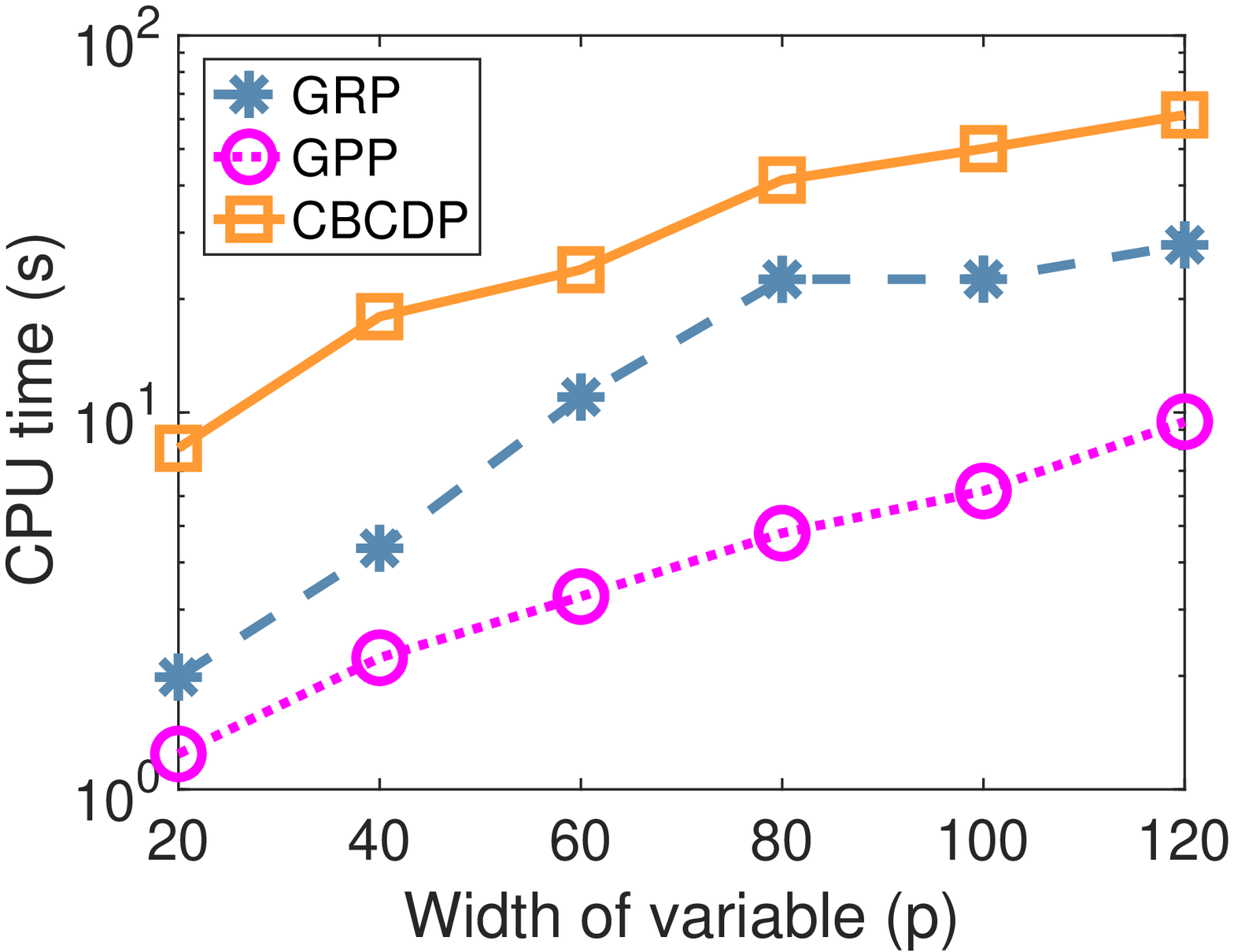}
			\end{minipage}
		}
		\subfloat[Function value variance (Problem 2)]{
			\label{subfig:width_2_fval}
			\begin{minipage}[t]{0.32\linewidth}
				\centering
				\includegraphics[width=1\textwidth]{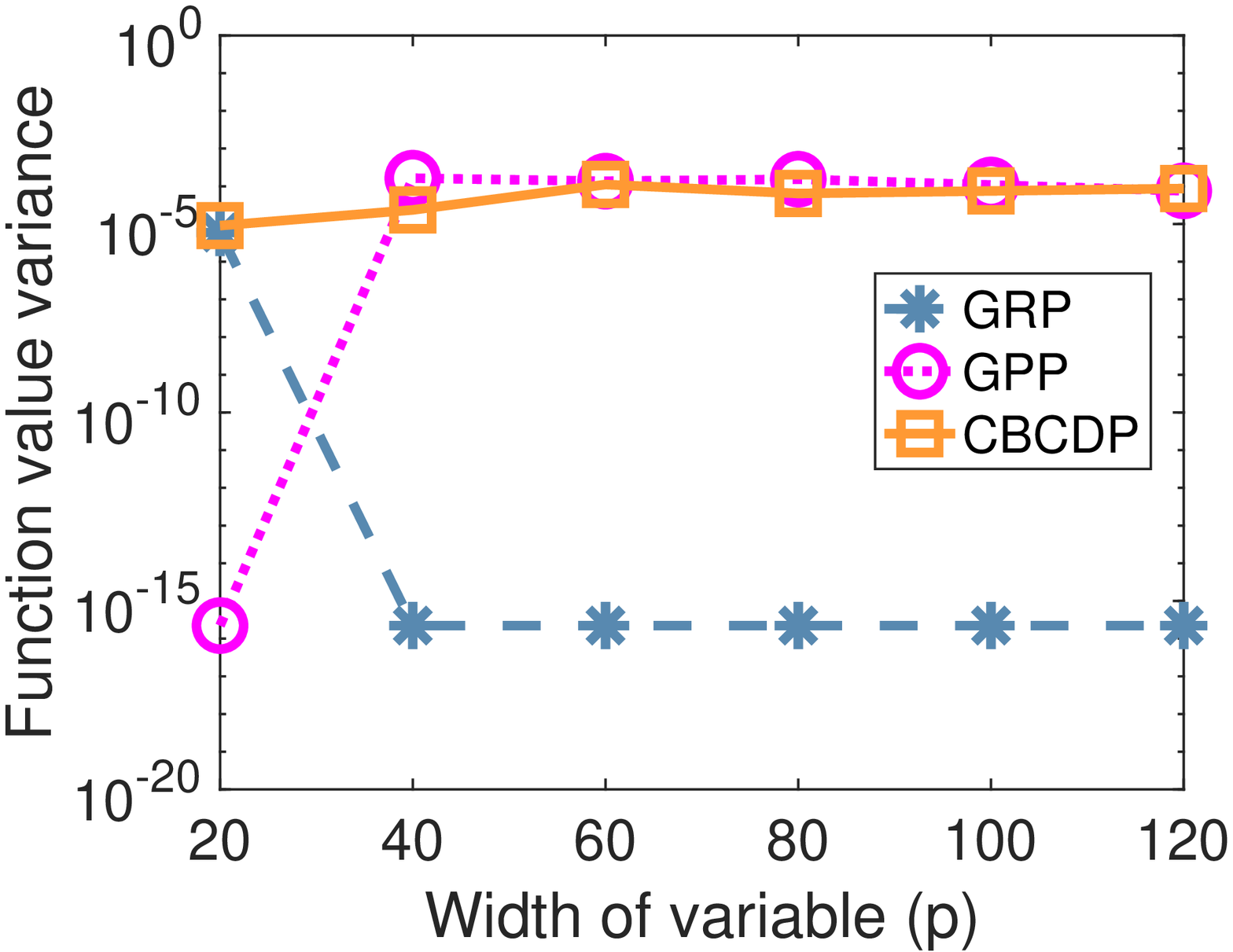}
			\end{minipage}
		}	
		\subfloat[KKT violation (Problem 2)]{
			\label{subfig:width_2_kkt}
			\begin{minipage}[t]{0.32\linewidth}
				\centering
				\includegraphics[width=1\textwidth]{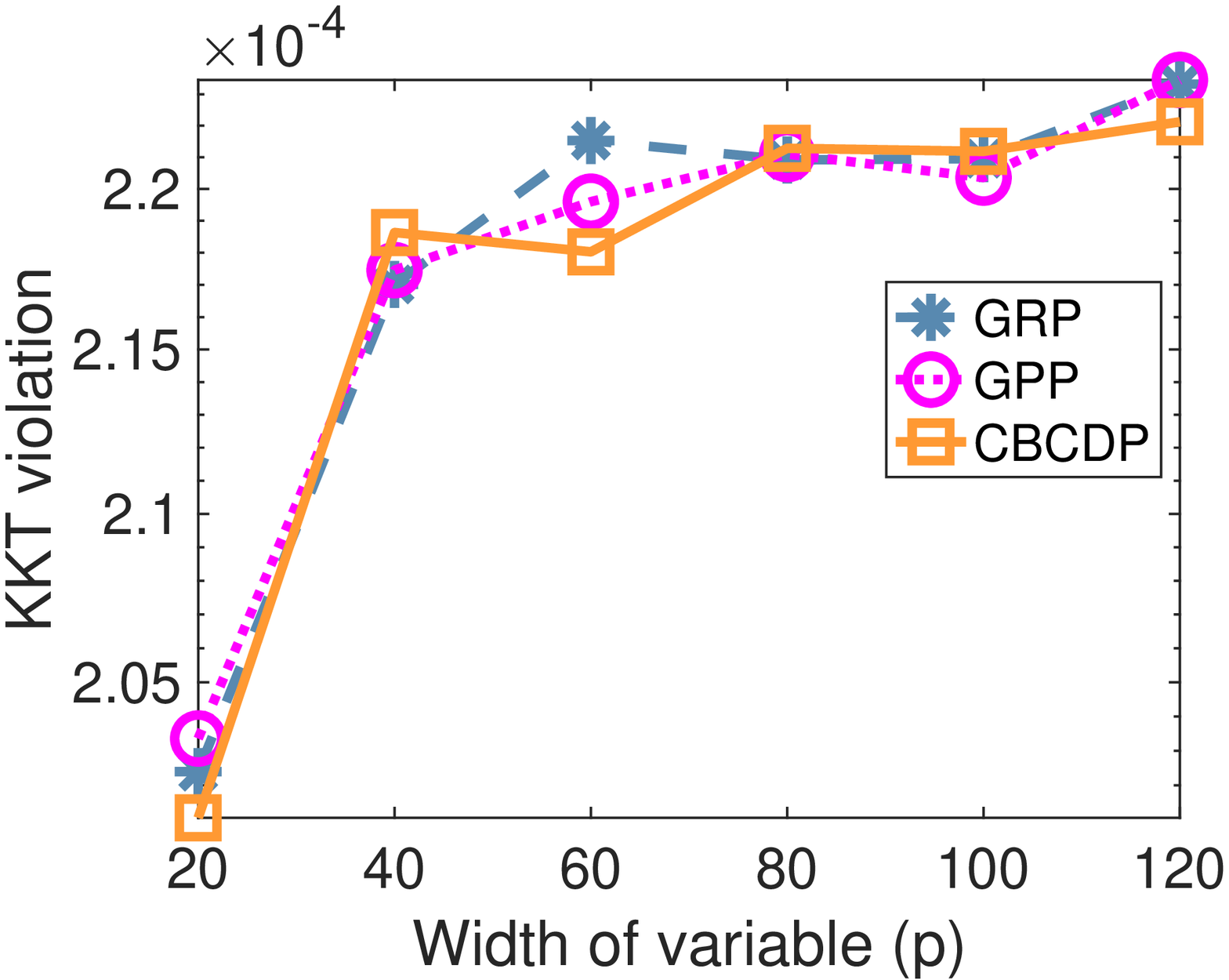}
			\end{minipage}
		}
		\caption{Comparison of GRP, GPP, and CBCDP for different $p$.}
		\label{fig:width}
	\end{figure}

\subsection{Performance comparison with other algorithms}

\label{subsec:pc}

	In this subsection, we compare the performance of GPP with other two state-of-the-art algorithms 
	for optimization problems over the Stiefel manifold.
	One is OptM\footnote{Downloadable from {\color{blue}\url{https://github.com/wenstone/OptM}}.}
	proposed in \cite{Wen2013}.
	The other one is MOptQR from the package 
	MANOPT\footnote{Downloadable from {\color{blue}\url{https://www.manopt.org/}}.}
	which is proposed in \cite{Absil2008}. 
	The original version is MOptQR-LS (manifold QR method with line search). 
	For fair comparison, we implement the same alternating BB step size strategy to MOptQR-LS, 
	which can significantly accelerate the algorithm as our GPP.
	
	We design five groups of testing problems based on Problem~2, 
	in each of which there is only one parameter varying with all the others fixed. 
	More specifically, we describe the varying parameters of each group as follows.
	\begin{enumerate}
		
		\item [$\bullet$] $n = 9000 + 1000j$ for $j = 1, 2, 3, 4, 5, 6$; $p = 100$.
		
		\item [$\bullet$] $p = 20j$ for $j = 1, 2, 3, 4, 5, 6$; $n = 10000$.
		
		\item [$\bullet$] $\beta = 1.1 + 0.3j$ for $j = 0, 1, 2, 3, 4, 5$; $n = 10000$; $p = 60$.
		 
		\item [$\bullet$] $\eta = 1.01 + 0.02j$ for $j = 0, 1, 2, 3, 4, 5$; $n = 10000$; $p = 60$.
		
		\item [$\bullet$] $\zeta = 1.01 + 0.05j$ for $j = 0, 1, 2, 3, 4, 5$; $n = 10000$; $p = 60$.
		
	\end{enumerate}
	All the other parameters take their default values.
	
	The numerical results of the above five groups of testing problems are depicted in 
	Figures \ref{fig:size_2} to \ref{fig:zeta_2}, respectively. 
	We observe that these algorithms achieve comparable KKT violation, 
	and GPP outperforms the other two algorithms in terms of CPU time and function value variance.
	
	\begin{figure}[t!]
		\centering
		\subfloat[CPU time (s)]{
			\label{subfig:size_time_2}
			\begin{minipage}[t]{0.32\linewidth}
				\centering
				\includegraphics[width=1\textwidth]{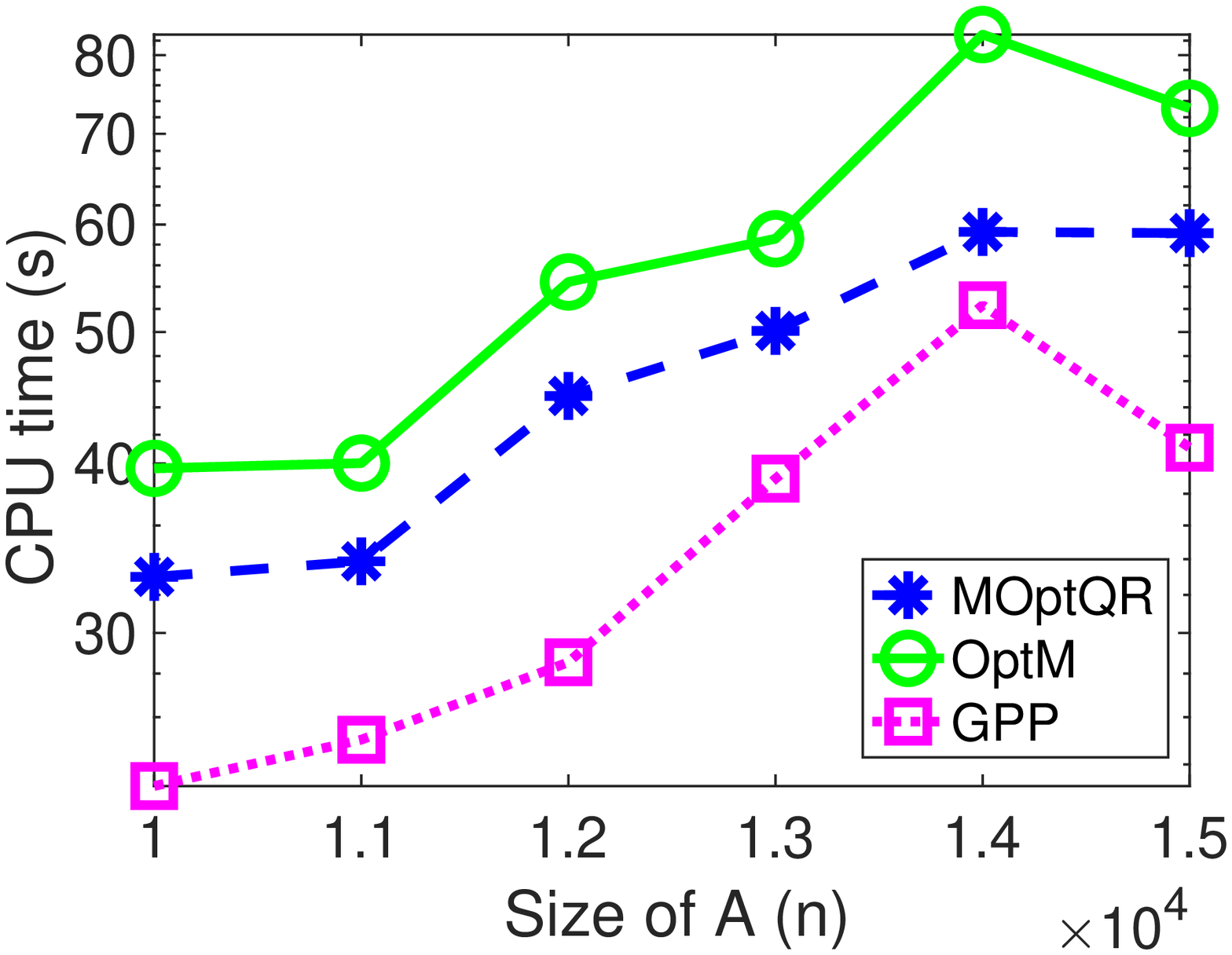}
			\end{minipage}
		}
		\subfloat[Function value variance]{
			\label{subfig:size_fval_2}
			\begin{minipage}[t]{0.32\linewidth}
				\centering
				\includegraphics[width=1\textwidth]{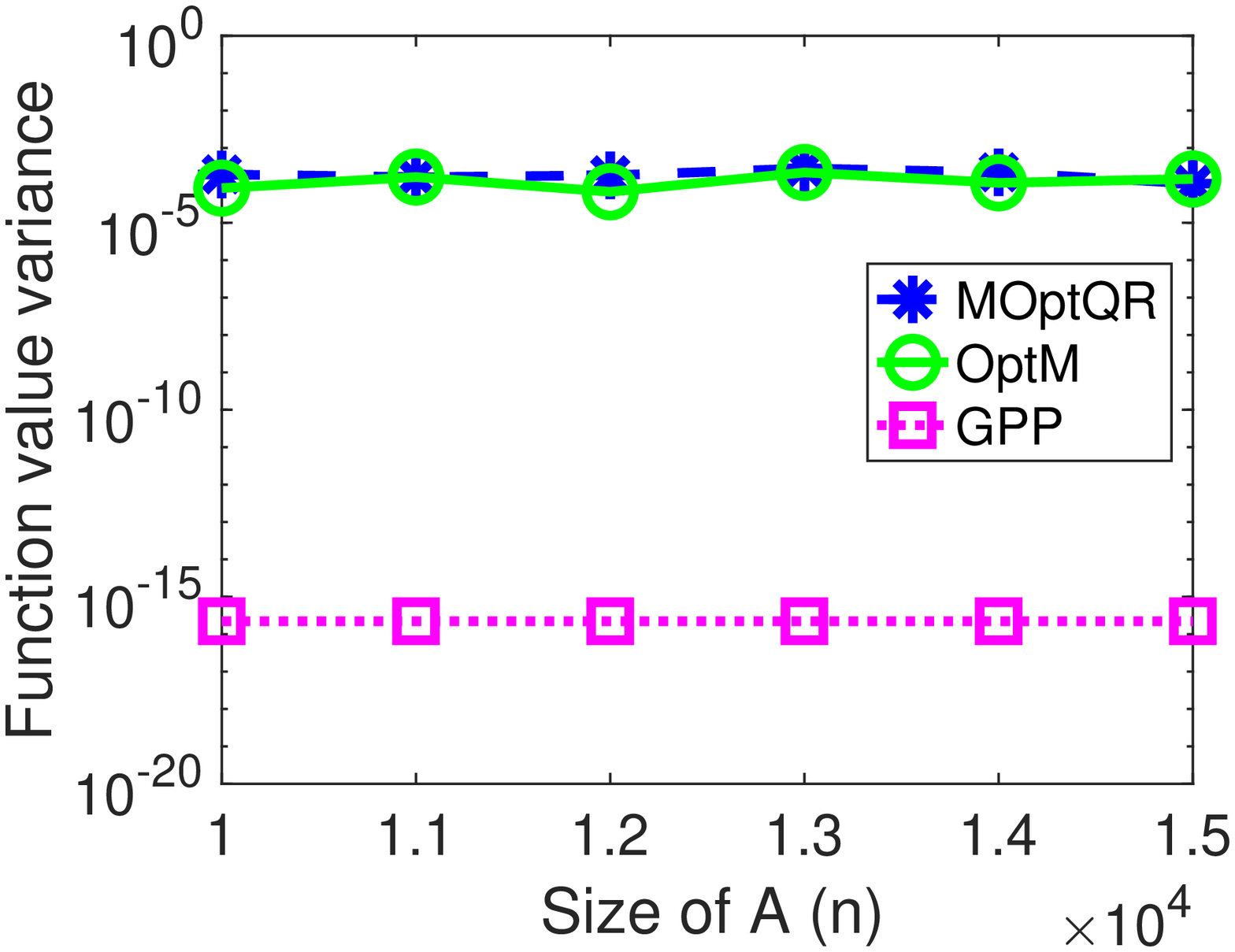}
			\end{minipage}
		}	
		\subfloat[KKT violation]{
			\label{subfig:size_kkt_2}
			\begin{minipage}[t]{0.32\linewidth}
				\centering
				\includegraphics[width=1\textwidth]{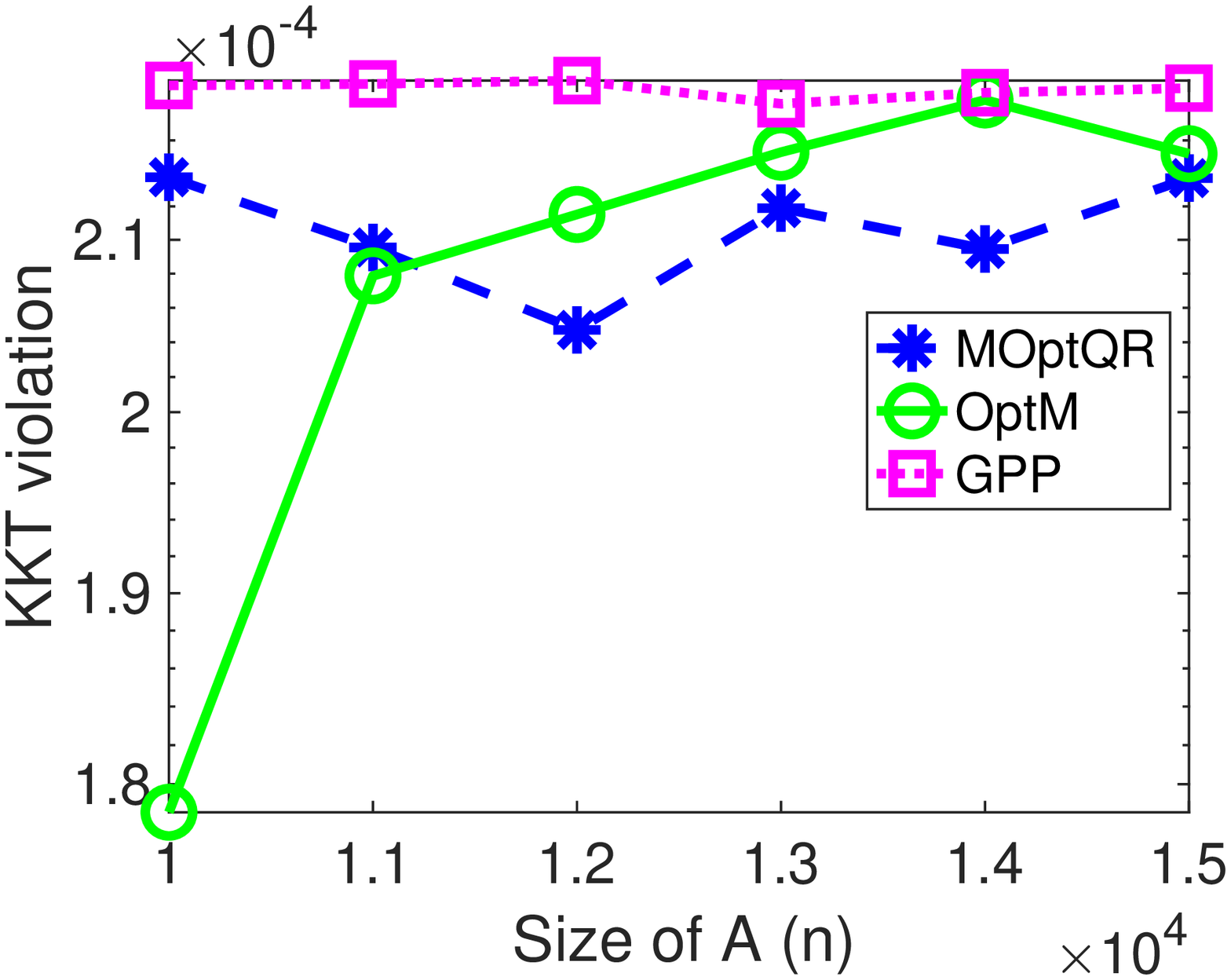}
			\end{minipage}
		}
		\caption{Comparison of GPP, OptM and MOptQR for different $n$ on Problem 2.}
		\label{fig:size_2}
	\end{figure}
	
	\begin{figure}[t!]
		\centering
		\subfloat[CPU time (s)]{
			\label{subfig:p_time_2}
			\begin{minipage}[t]{0.32\linewidth}
				\centering
				\includegraphics[width=1\textwidth]{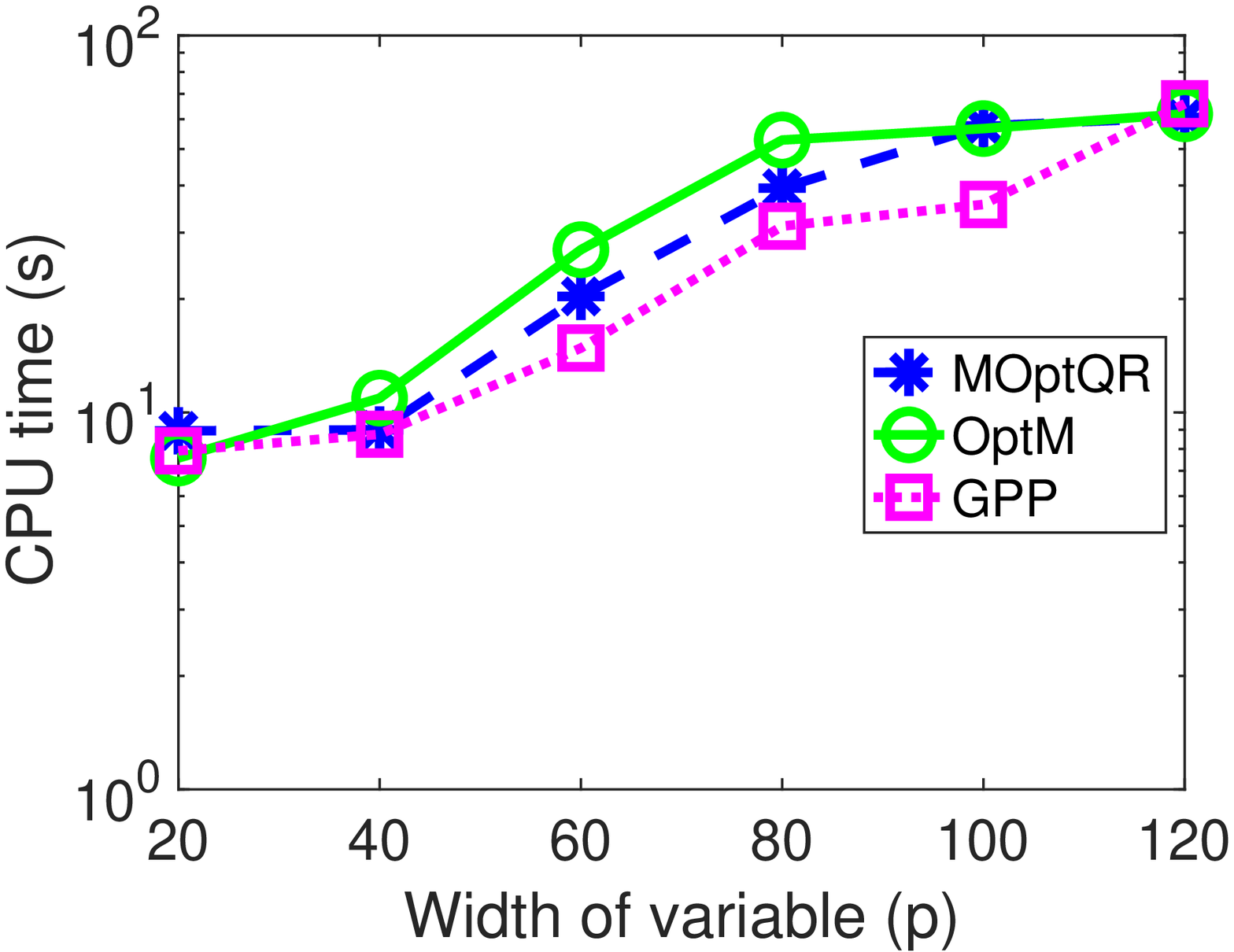}
			\end{minipage}
		}
		\subfloat[Function value variance]{
			\label{subfig:p_fval_2}
			\begin{minipage}[t]{0.32\linewidth}
				\centering
				\includegraphics[width=1\textwidth]{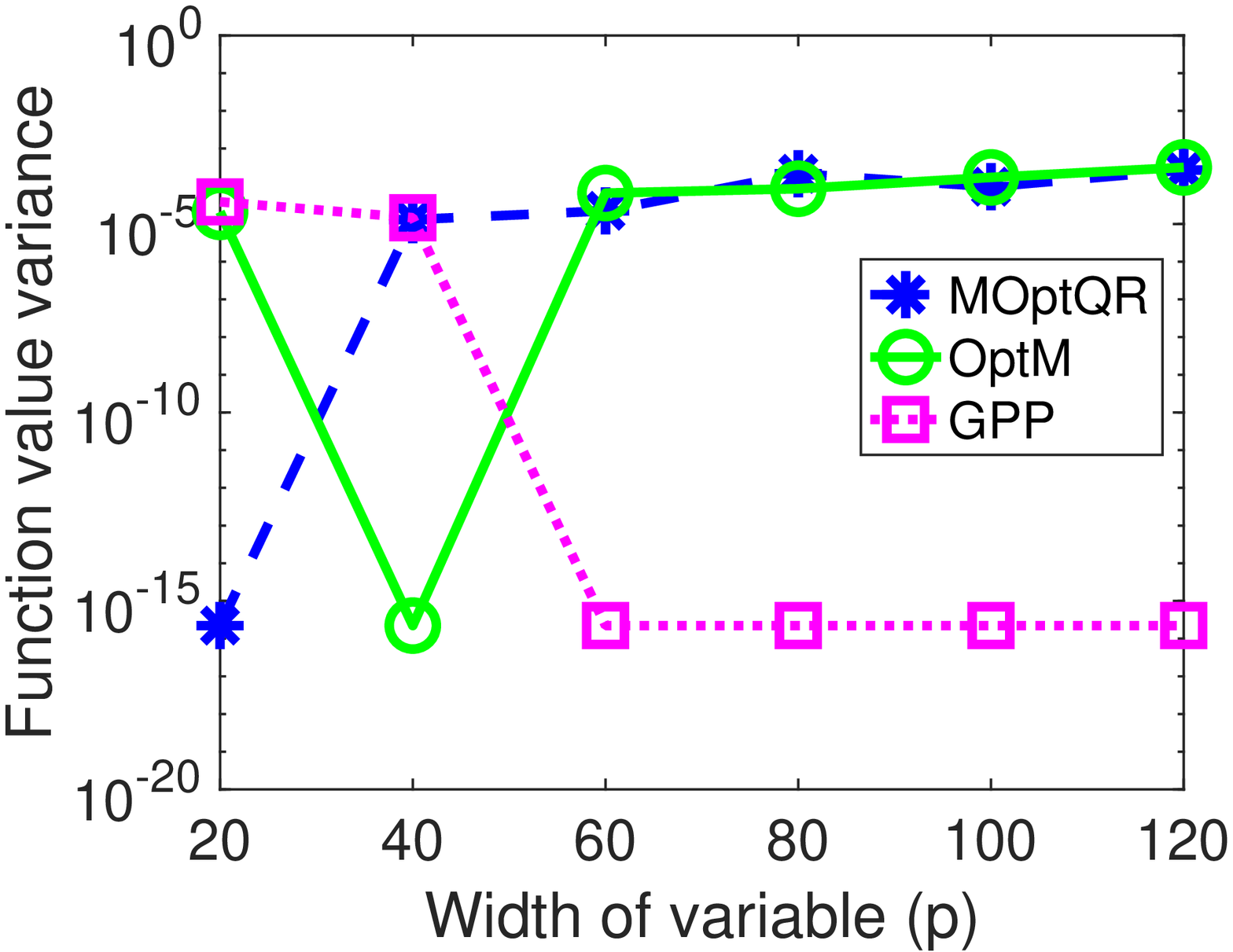}
			\end{minipage}
		}	
		\subfloat[KKT violation]{
			\label{subfig:p_kkt_2}
			\begin{minipage}[t]{0.32\linewidth}
				\centering
				\includegraphics[width=1\textwidth]{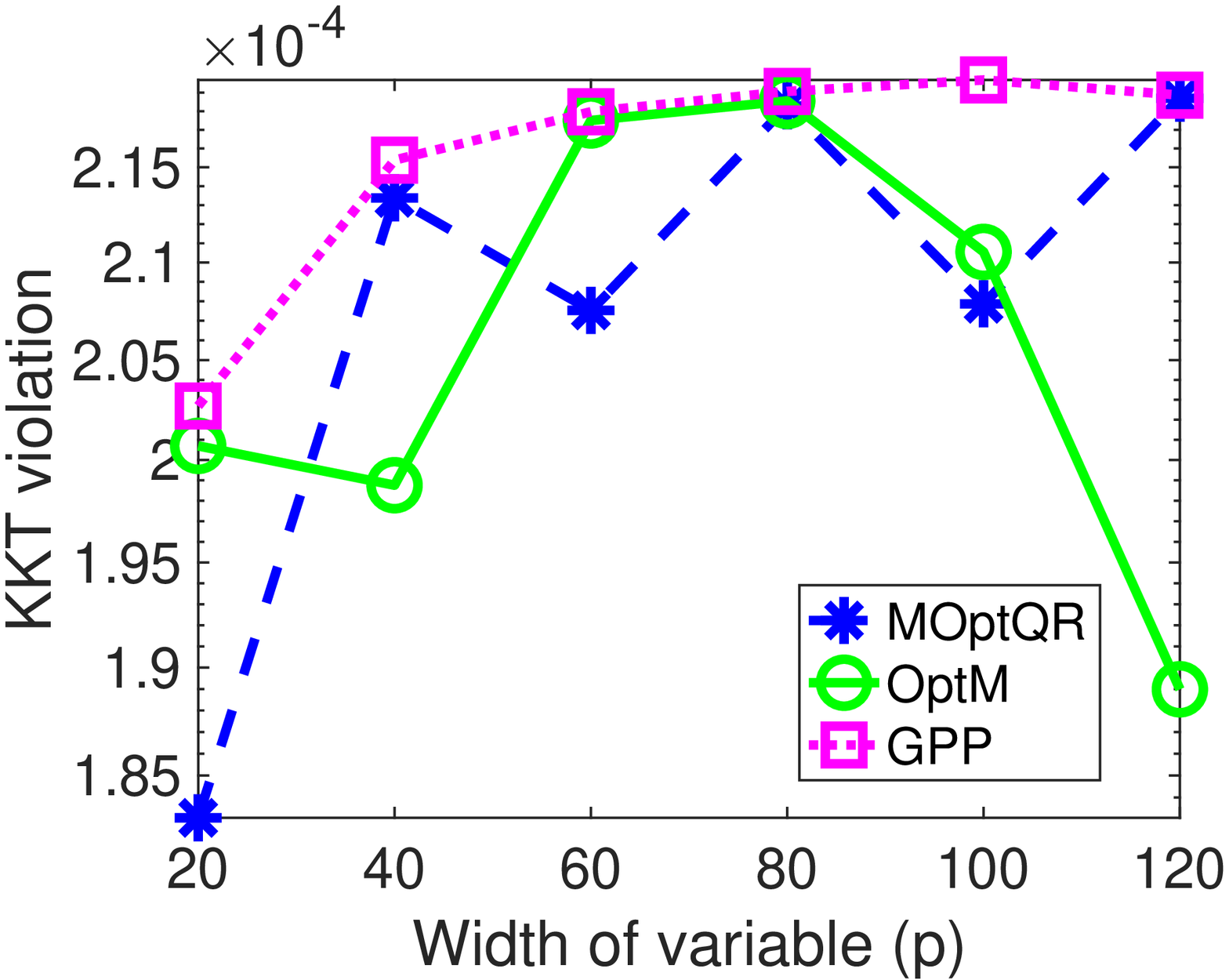}
			\end{minipage}
		}
		\caption{Comparison of GPP, OptM and MOptQR for different $p$ on Problem 2.}
		\label{fig:p_2}
	\end{figure}
	
	\begin{figure}[t!]
		\centering
		\subfloat[CPU time (s)]{
			\label{subfig:beta_time_2}
			\begin{minipage}[t]{0.32\linewidth}
				\centering
				\includegraphics[width=1\textwidth]{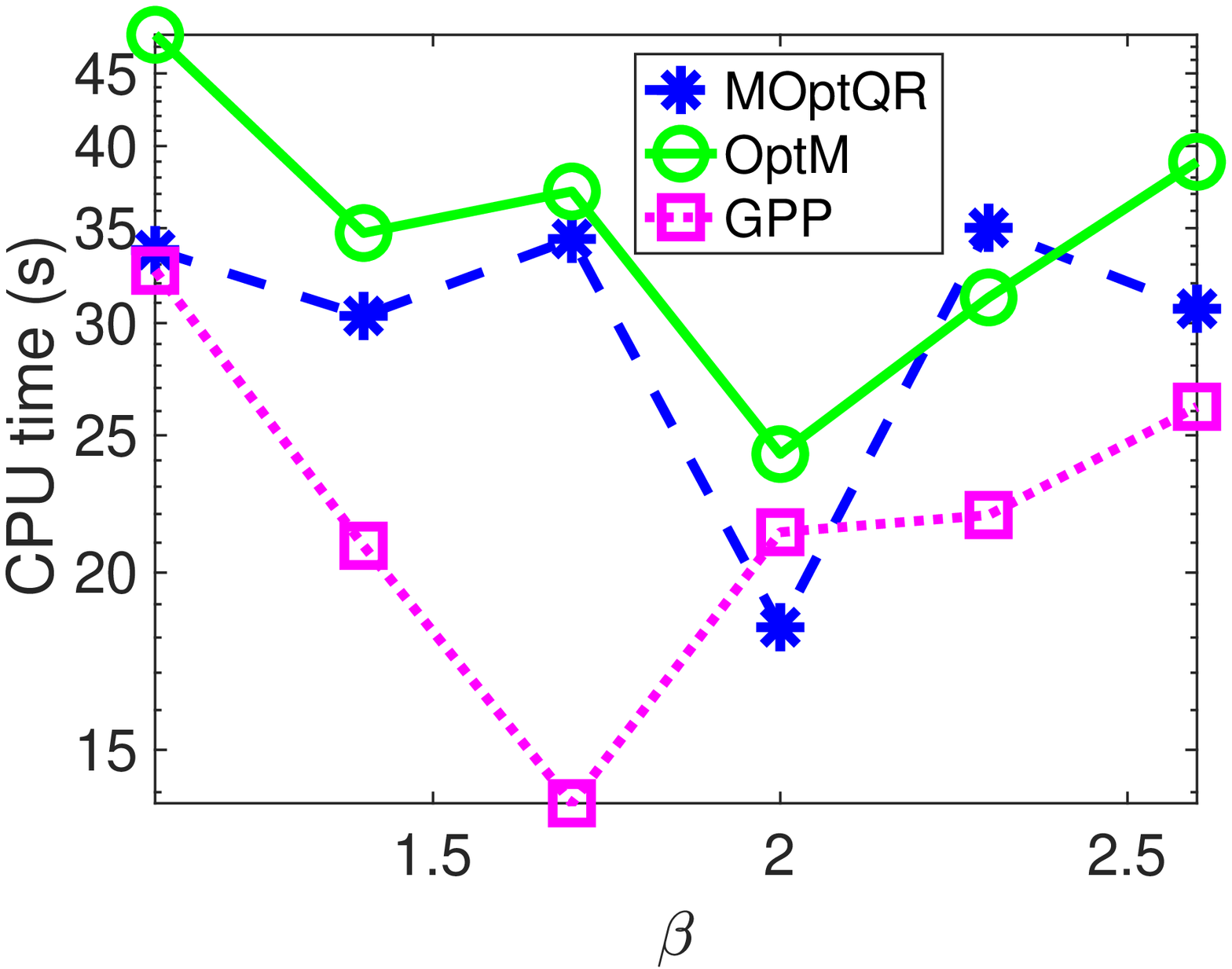}
			\end{minipage}
		}
		\subfloat[Function value variance]{
			\label{subfig:beta_fval_2}
			\begin{minipage}[t]{0.32\linewidth}
				\centering
				\includegraphics[width=1\textwidth]{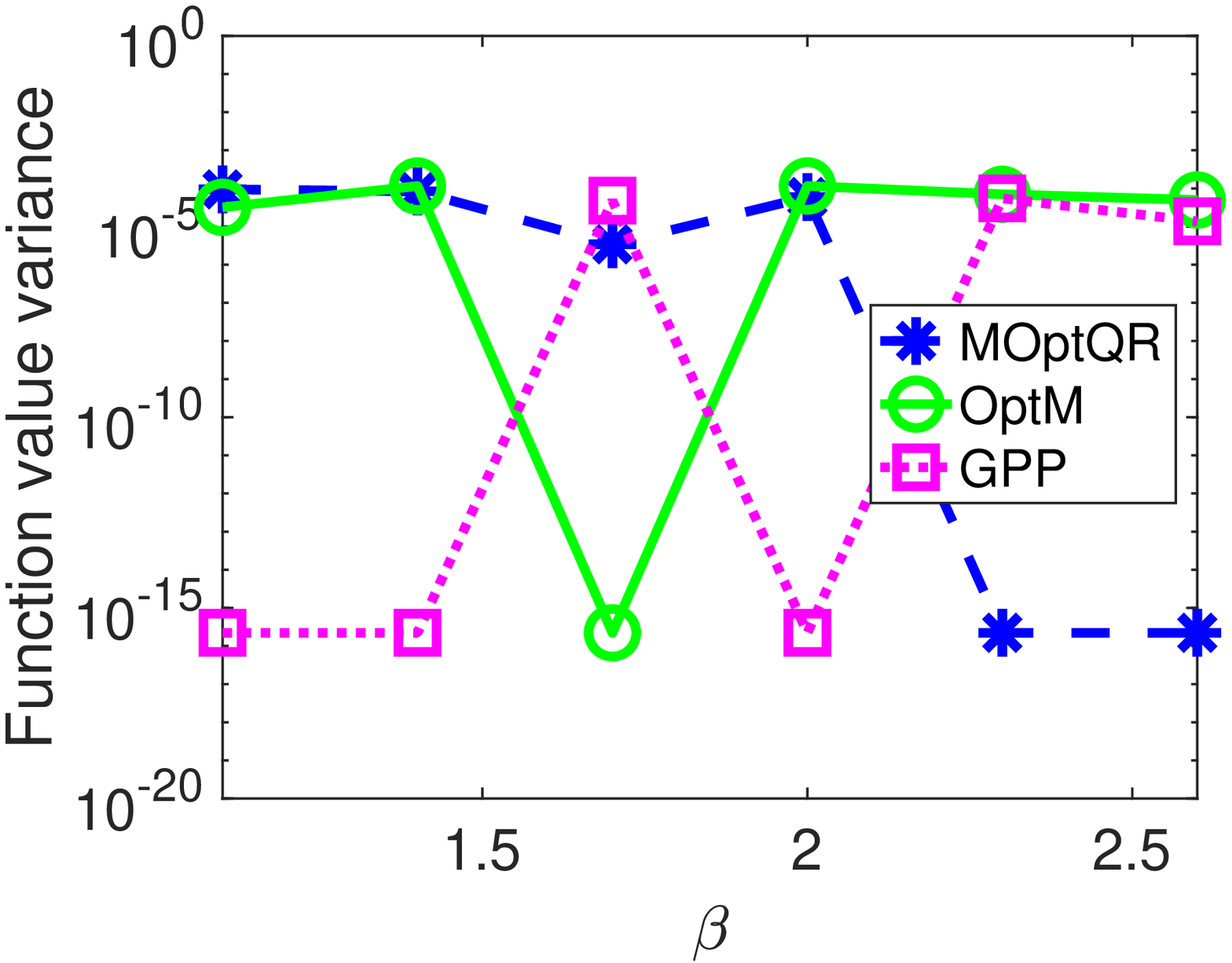}
			\end{minipage}
		}	
		\subfloat[KKT violation]{
			\label{subfig:beta_kkt_2}
			\begin{minipage}[t]{0.32\linewidth}
				\centering
				\includegraphics[width=1\textwidth]{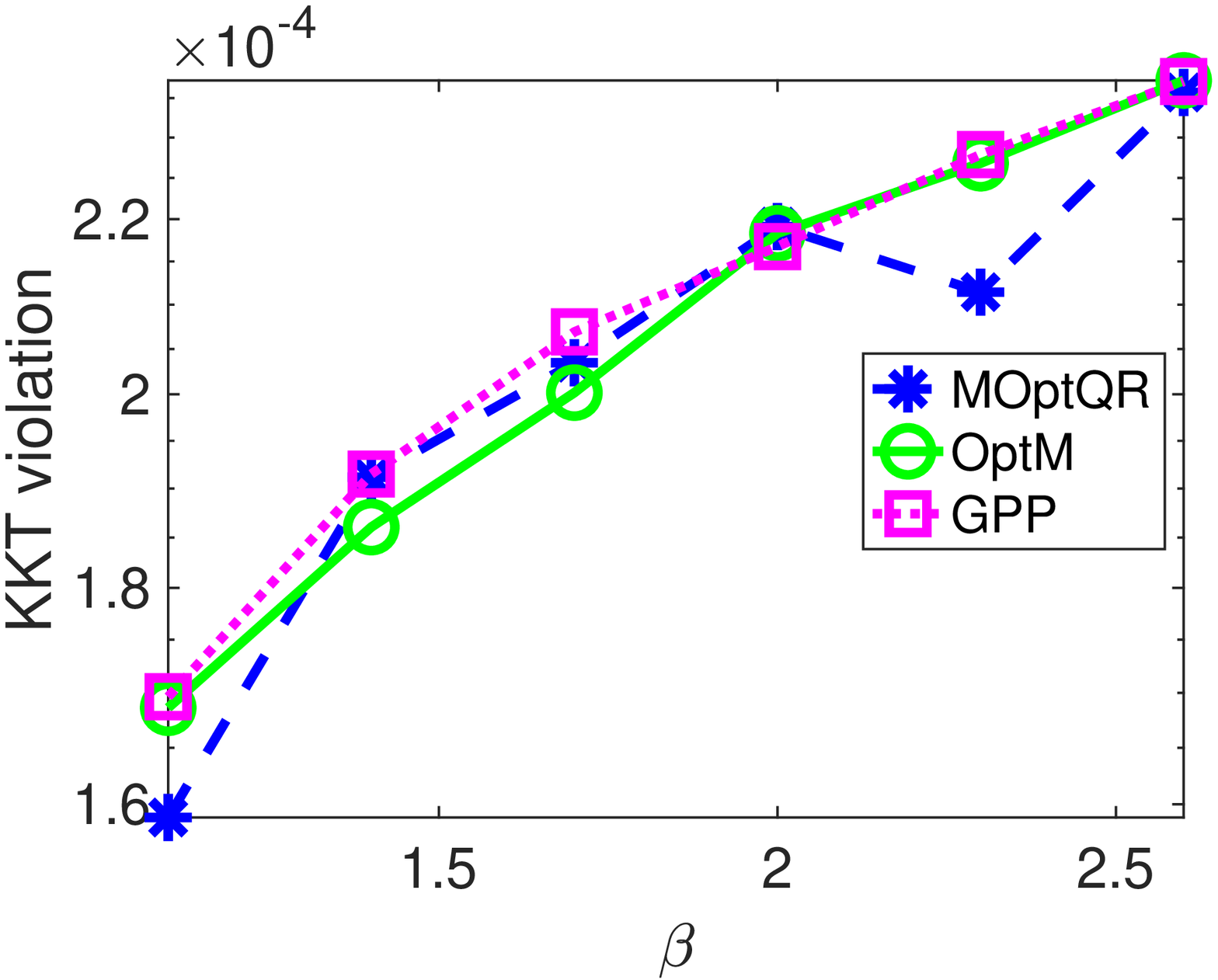}
			\end{minipage}
		}
		\caption{Comparison of GPP, OptM and MOptQR for different $\beta$ on Problem 2.}
		\label{fig:beta_2}
	\end{figure}
	
	\begin{figure}[t!]
		\centering
		\subfloat[CPU time (s)]{
			\label{subfig:eta_time_2}
			\begin{minipage}[t]{0.32\linewidth}
				\centering
				\includegraphics[width=1\textwidth]{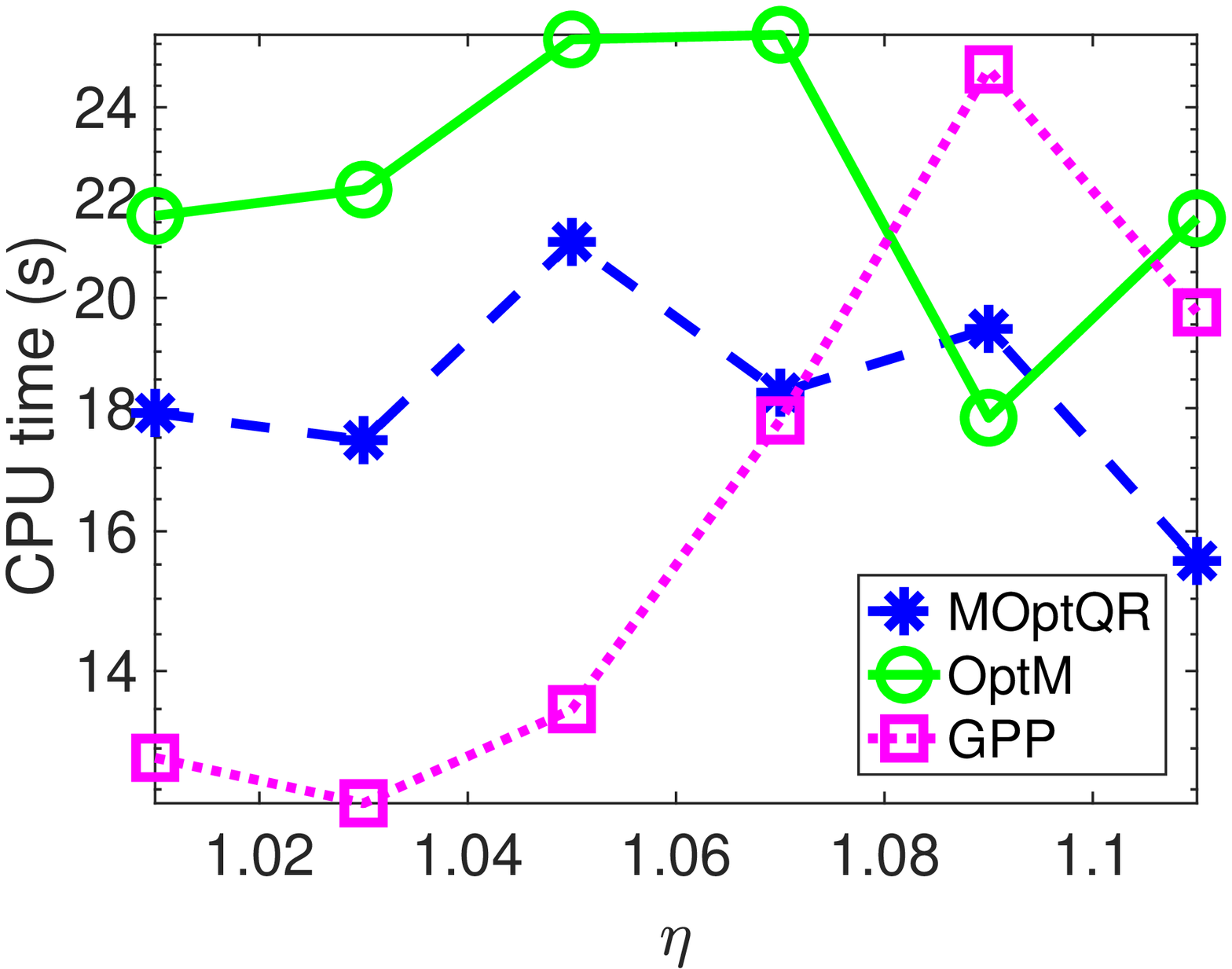}
			\end{minipage}
		}
		\subfloat[Function value variance]{
			\label{subfig:eta_fval_2}
			\begin{minipage}[t]{0.32\linewidth}
				\centering
				\includegraphics[width=1\textwidth]{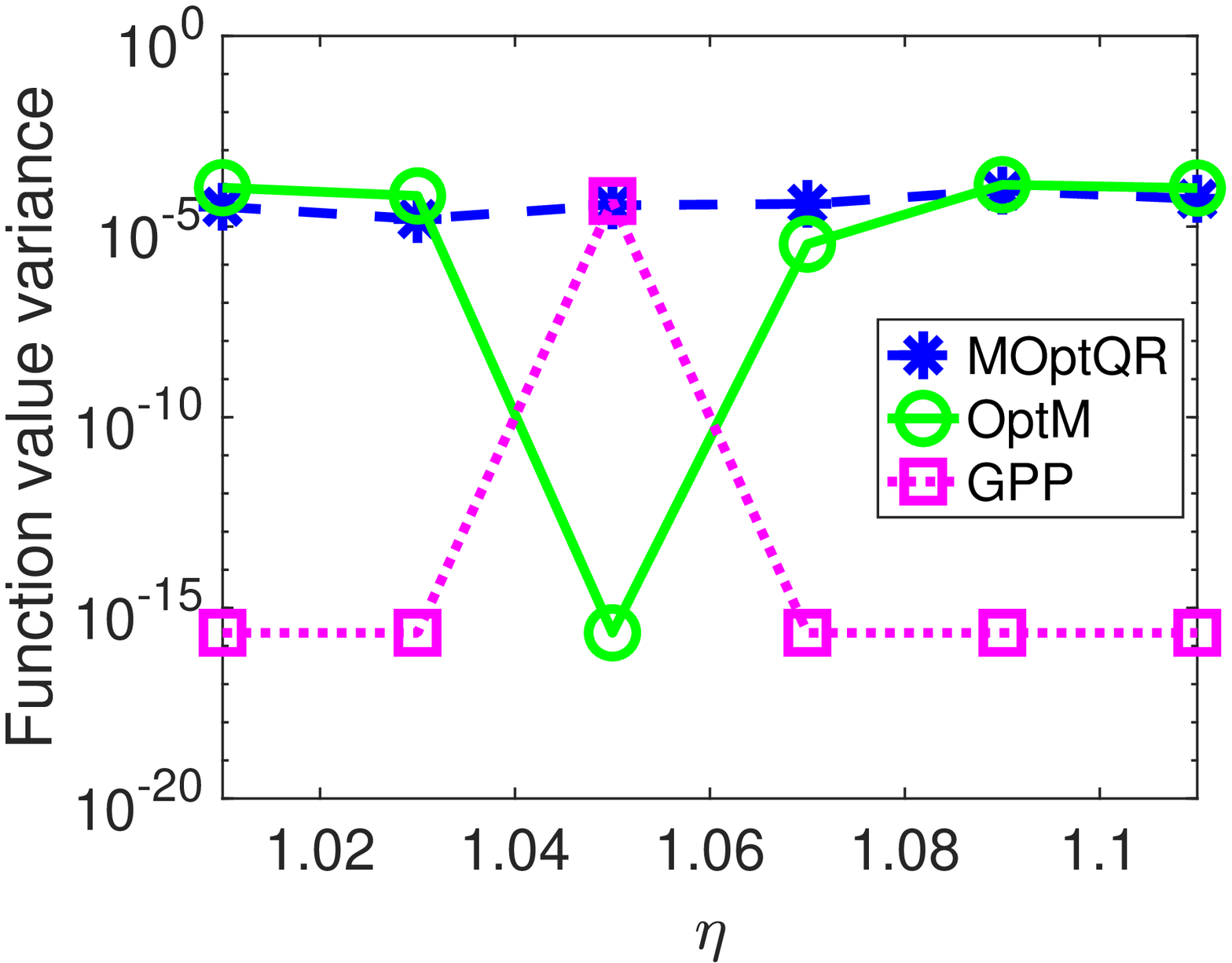}
			\end{minipage}
		}	
		\subfloat[KKT violation]{
			\label{subfig:eta_kkt_2}
			\begin{minipage}[t]{0.32\linewidth}
				\centering
				\includegraphics[width=1\textwidth]{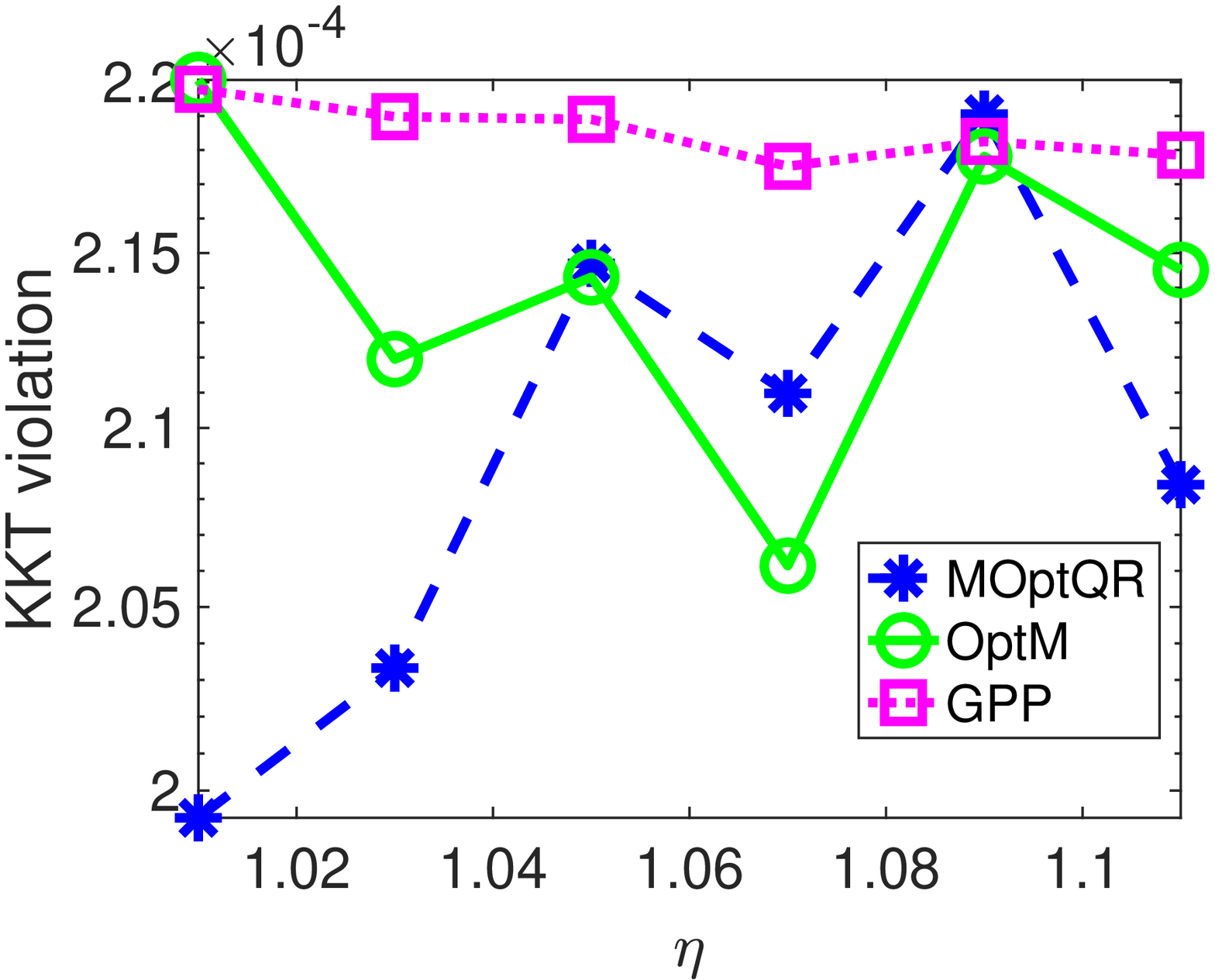}
			\end{minipage}
		}
		\caption{Comparison of GPP, OptM and MOptQR for different $\eta$ on Problem 2.}
		\label{fig:eta_2}
	\end{figure}
	
	\begin{figure}[t!]
		\centering
		\subfloat[CPU time (s)]{
			\label{subfig:zeta_time_2}
			\begin{minipage}[t]{0.32\linewidth}
				\centering
				\includegraphics[width=1\textwidth]{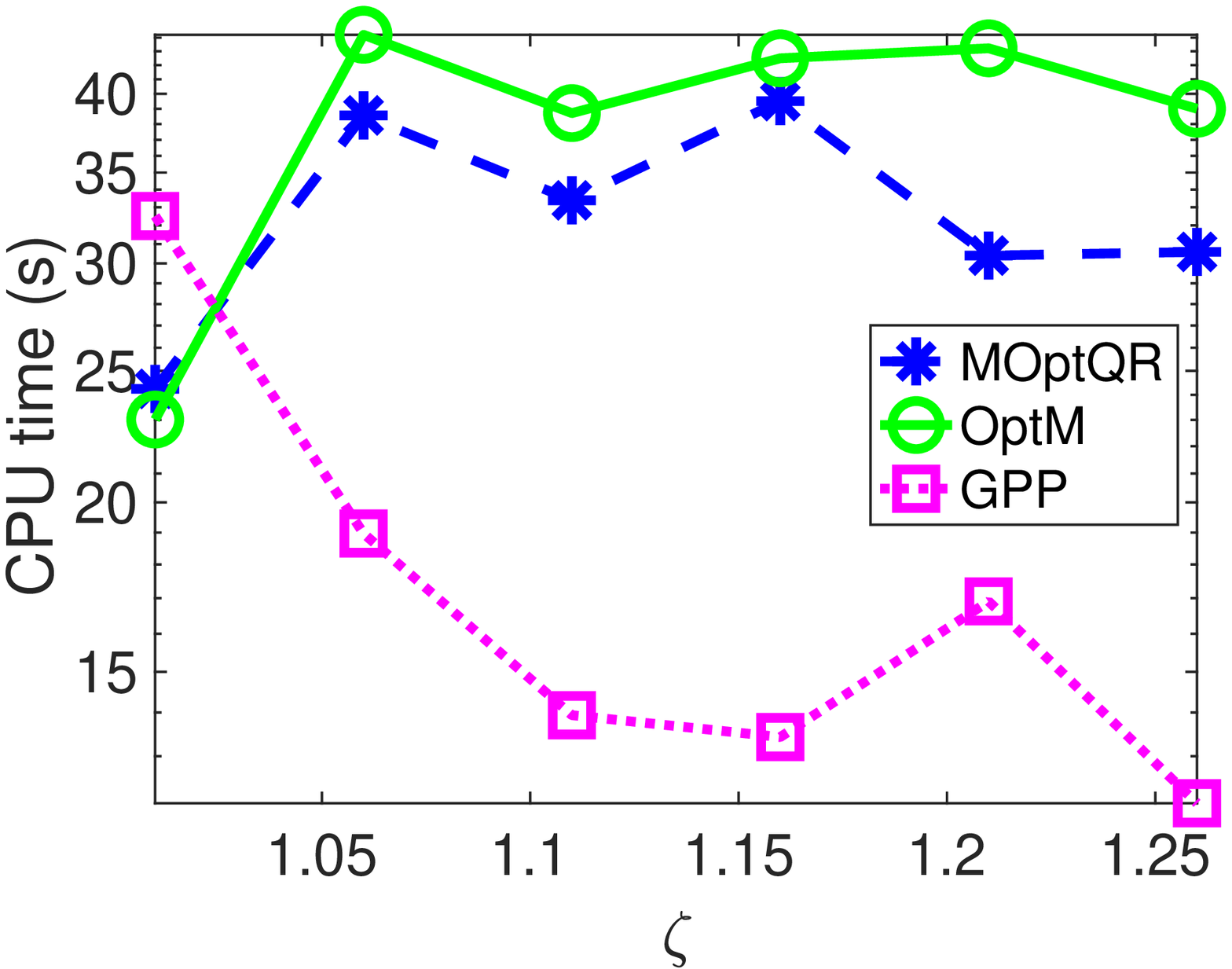}
			\end{minipage}
		}
		\subfloat[Function value variance]{
			\label{subfig:zeta_fval_2}
			\begin{minipage}[t]{0.32\linewidth}
				\centering
				\includegraphics[width=1\textwidth]{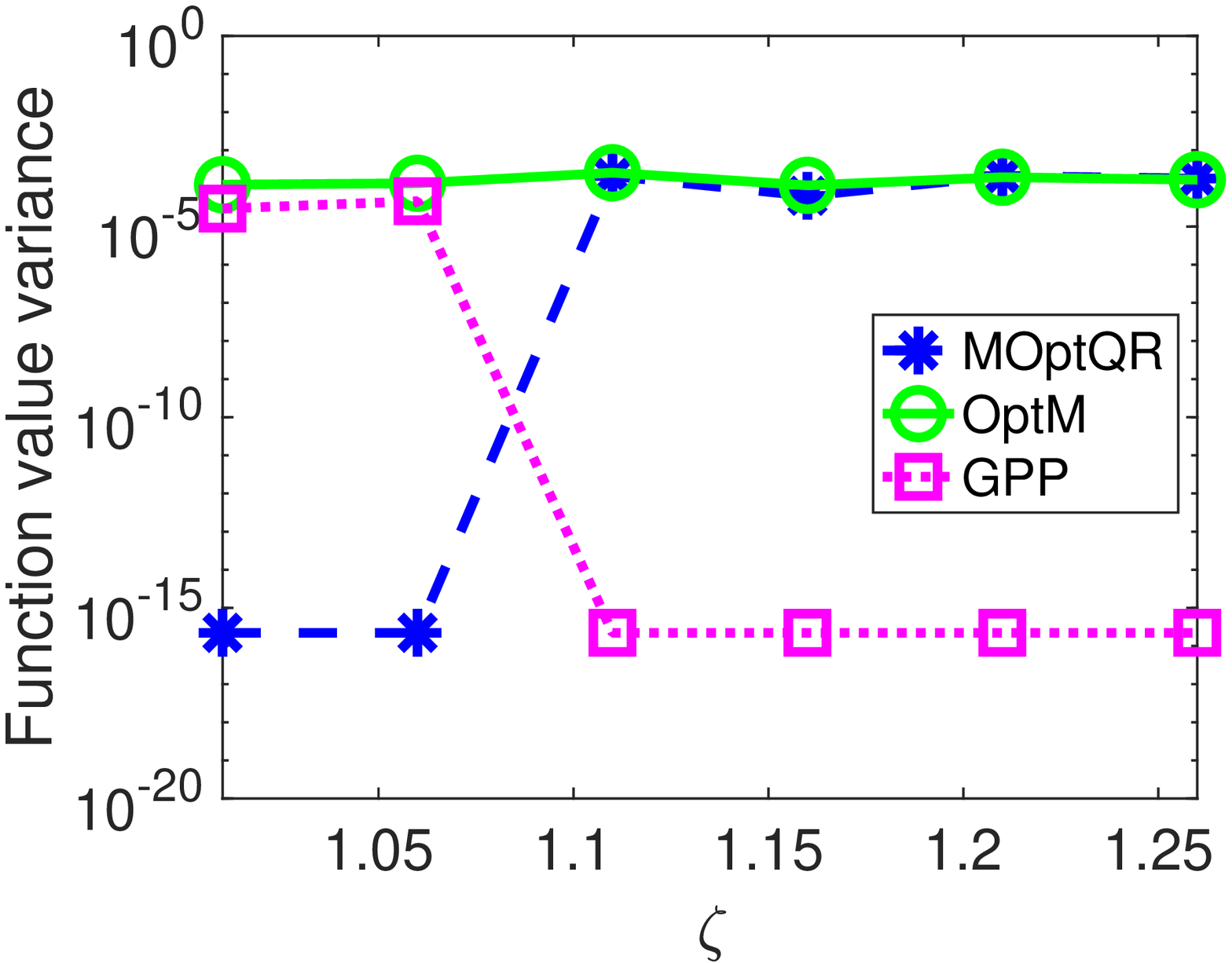}
			\end{minipage}
		}	
		\subfloat[KKT violation]{
			\label{subfig:zeta_kkt_2}
			\begin{minipage}[t]{0.32\linewidth}
				\centering
				\includegraphics[width=1\textwidth]{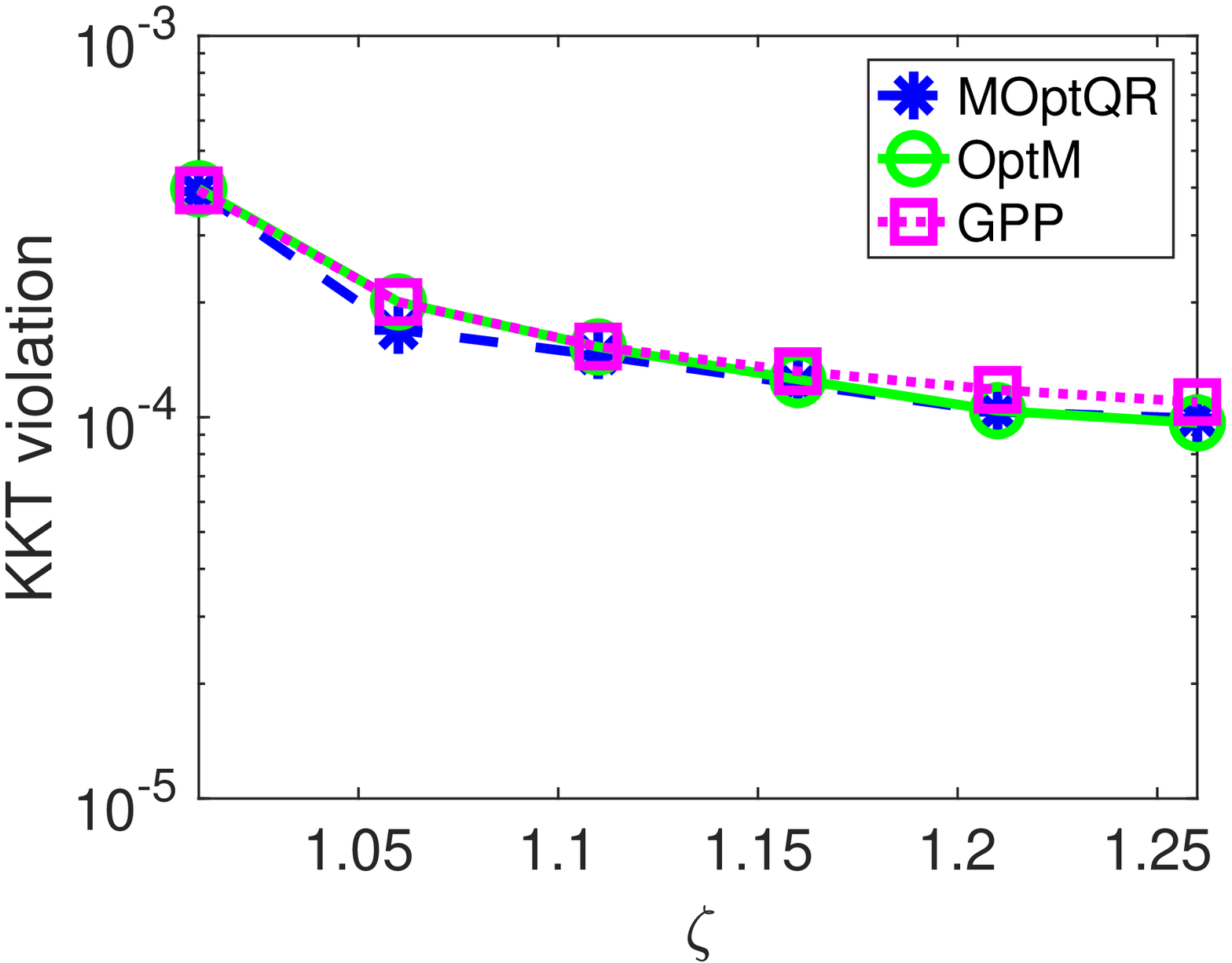}
			\end{minipage}
		}
		\caption{Comparison of GPP, OptM and MOptQR for different $\zeta$ on Problem 2.}
		\label{fig:zeta_2}
	\end{figure}

	In order to make a more comprehensive comparison, 
	we use performance profiles based on \cite{Dolan2002} 
	to visualize the different behaviors among these solvers. 
	For this purpose, we design a variety of random problems based on Problem 2,
	which can be described as follows:
	\begin{enumerate}
		
		\item [$\bullet$] $n = 2000 + 1000j$ for $j = 1, 2, 3, 4, 5, 6$;	
		
		\item [$\bullet$] $p = 20j$ for $j = 1, 2, 3, 4, 5, 6$;
		
		\item [$\bullet$] $\beta = 1 + 0.5j$ for $j = 0, 1, 2, 3$;	
		
		\item [$\bullet$] $\eta = 1.01 + 0.05j$ for $j = 0, 1, 2, 3$;	
		
		\item [$\bullet$] $\zeta = 1.1 + 0.05j$ for $j = 0, 1, 2, 3$.
		
	\end{enumerate}

	There are altogether $6 \times 6 \times 4 \times 4 \times 4 = 2304$ randomly generated problems. 
	We simply explain the performance profile as the following. 
	For problem $m$ and solver $s$, we use $t_{m,s}$ to represent its CPU time. 
	Performance ratio is defined as $r_{m,s} = t_{m,s} / \min_s \hkh{ t_{m,s} }$. 
	If solver $s$ fails to solve problem $m$, the ratio $r_{m,s}$ is set to a preset large number. 
	Finally, the overall performance of solver $s$ is defined by
	\begin{equation*}
		\pi_s (\omega) = \dfrac{\mbox{number of problems where $r_{m,s} \leq \omega$}}{\mbox{total number of problems}}.
	\end{equation*}
	It means the percentage of testing problems that can be solved in $\omega \min_s \hkh{t_{m,s}}$ seconds. 
	It is clear that the closer $\pi_s$ is to 1, the better performance solver $s$ has. 
	
	The performance profile with respect to the CPU time is given in Figure~\ref{fig:perf_2}. 
	On the $2304$ testing problems, GPP is of the best numerical behavior in terms of CPU time 
	and it always solves problems in no more than twice the fastest time among these three algorithms.  
	In addition, we also provide the average KKT violation, feasibility violation and function value variance 
	over these $2304$ random problems in Table~\ref{tb:perf_2}, 
	which shows that all solvers achieve a comparable average KKT violation, 
	feasibility violation, and function value variance.


	\begin{figure}[htbp]
		\centering
		\includegraphics[width=0.5\linewidth]{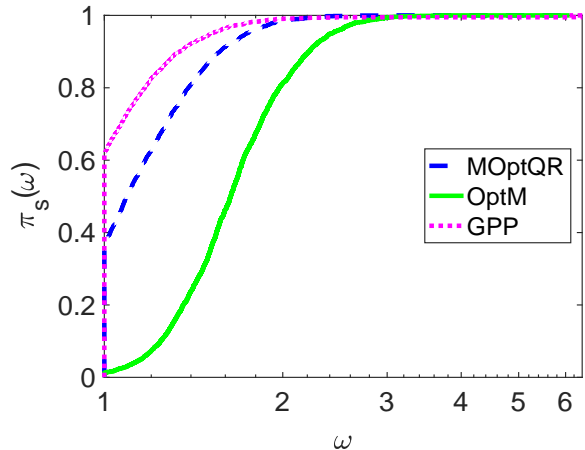}
		\caption{Performance profile on $2304$ problems with respect to CPU time.}
		\label{fig:perf_2}
	\end{figure}
	
	\begin{table}[tbp]
		\centering
		\begin{tabular}{cccc}
			\toprule
			&  GPP & MOptQR & OptM  \\
			\midrule
			KKT violation           &  $1.4917\times 10^{-3}$ & $1.1534\times 10^{-3}$ & $1.1516\times 10^{-3}$ \\
			Function value variance & $3.3934\times 10^{-4}$ & $6.8694\times
			10^{-4}$ & $8.1899\times 10^{-4}$ \\
			Feasibility violation   & $2.5227\times 10^{-15}$ & $2.2282\times 10^{-15}$ & $2.0217\times 10^{-15}$ \\ 
			\bottomrule    	                          
		\end{tabular}
		\caption{Average KKT violation, feasibility violation, and function value variance.}
		\label{tb:perf_2}
	\end{table}

\section{Conclusion}

\label{sec:c}

	The first-order algorithmic framework proposed in \cite{Gao2018}
	consists of a function value reduction step in the Euclidean space
	and a rotation step to guarantee the symmetry of
	the explicit expression of Lagrangian multipliers associate with orthogonality constraints.
	Three algorithms based on this framework have illustrated their efficiency 
	in solving problems such as minimizing quadratic objective over the Stiefel manifold 
	and discretized Kohn--Sham total energy minimization. 
	However, a crucial limitation of this approach is its strict assumption on the objective. 
	In practice, there are quite some critical instances that do not satisfy that assumption.
	
	In this paper, we propose a novel multipliers correction strategy,
	which minimizes a linear approximation with a proximal term 
	in the range space of the intermediate iterate generated by the function value reduction step. 
	Such correction strategy can guarantee further function value reduction 
	in proportion to the ``symmetry'' violation. 
	Consequently, the convergent point satisfies the ``symmetry" property.  
	We establish the complete global convergence analysis and worst case complexity as well. 
	Furthermore, numerical experiments illustrate that the new multipliers correction methods 
	have better performances than those proposed in \cite{Gao2018}. 
	Remarkably, our multipliers correction methods can solve problems 
	that those proposed in \cite{Gao2018} can not solve.
	In solving these testing problems, our methods outperform other state-of-the-art first-order approaches.

\bibliographystyle{abbrv}

\bibliography{MCM}

\end{document}